\documentclass[11pt]{article}
\usepackage{amsmath,amsfonts,amssymb, amsthm,enumerate,graphicx}
\usepackage{tikz,authblk}
\usepackage{subfig}
\usepackage{url}

\usepackage{stackengine,scalerel}
\usetikzlibrary {decorations.pathmorphing, decorations.pathreplacing, decorations.shapes}

\usetikzlibrary{arrows.meta}

\usepackage{hyperref}

\newtheorem{theorem} {{\textsf{Theorem}}}
\newtheorem{proposition}[theorem]{{\textsf{Proposition}}}
\newtheorem{corollary}[theorem]{{\textsf{Corollary}}}

\newtheorem{definition}[theorem]{{\textsf{Definition}}}

\newtheorem{lemma}[theorem]{{\textsf{Lemma}}}

\def\B{{\mathcal{B}}}

\def\S{{\mathbb{S}^{n}}}

\def\D{\Delta}

\def\G{\Gamma}

\begin{document}
\title{Minimal simplicial degree $d$ self-maps of $\mathbb{S}^{n-1}\times \mathbb{S}^1$}

\author{Anshu Agarwal$^1$, Biplab Basak$^{1,2}$, and Sourav Sarkar$^3$}
\date{7 July, 2026}
\maketitle
\vspace{-10mm}

\begin{center}
\noindent {\small Department of Mathematics, IIT Delhi, Hauz Khas, New Delhi 110016$^1$.}

\noindent {\small Institute of Mathematical Sciences, Pusan National University, South Korea$^{3}$.}

\noindent {\small {\em E-mail addresses:} \url{maz228084@maths.iitd.ac.in} (A. Agarwal), \url{biplab@iitd.ac.in} (B. Basak), \url{sarkarsourav610@gmail.com} (S. Sarkar).}

\footnotetext[2]{Corresponding author}

\end{center}

\hrule
	
\begin{abstract}
The degree of a map between orientable manifolds is a fundamental concept in topology, providing important information about the structure of manifolds and the behavior of maps between them. A simplicial cell complex $K$ is called a \emph{colored triangulation} of a closed PL $n$-manifold $M$ if the $1$-skeleton of $K$ admits a proper vertex-coloring with $n+1$ colors and $|K|$ is PL-homeomorphic to $M$.

In this article, we construct, for every $d \in \mathbb{Z}$ and $n \geq 2$, a degree $d$ simplicial map from a $(2(n+1)\max\{|d|,1\})$-facet colored triangulation of $\mathbb{S}^{n-1} \times \mathbb{S}^1$ to the standard $2(n+1)$-facet colored triangulation of $\mathbb{S}^{n-1} \times \mathbb{S}^1$. Additionally, for every $d \in \mathbb{Z}$ and $n \geq 2$, we construct a degree $d$ simplicial map from a $(2\max\{|d|,1\})$-facet colored triangulation of $\mathbb{S}^n$ to the standard $2$-facet colored triangulation of $\mathbb{S}^n$.

For $M = \mathbb{S}^{n-1} \times \mathbb{S}^1$ and $\mathbb{S}^n$, with $n \geq 2$, these simplicial degree $d$ self-maps of $M$ are minimal with respect to their standard colored triangulations, in the sense that there does not exist a colored triangulation $\mathcal{K}$ of $M$ with fewer facets than the constructed one that admits a simplicial map $f : \mathcal{K} \to \mathcal{K}_M$ of degree $d$, where $\mathcal{K}_M$ denotes the standard colored triangulation of $M$.
\end{abstract}

\noindent {\small {\em MSC 2020\,:} Primary 57Q15; Secondary  05C15, 55M25.
		
}
	
\medskip

\section{Introduction}
 The notion of the degree of a map between orientable manifolds, which goes back to the study of Brouwer \cite{Brower1911}, is a powerful invariant that bridges various areas of topology, geometry, and mathematical physics. It provides deep insights into the structure and behavior of manifolds and the maps between them. For orientable manifolds, the degree indicates whether the map preserves or reverses orientation. A positive degree signifies orientation preservation, while a negative degree signifies orientation reversal. In mathematical physics, the degree can represent quantities such as the winding number, topological charge, or the flux of a field through a manifold. These quantities often have physical interpretations, such as in the study of solitons and instantons.
 
Extensive research has been conducted on the degree of maps between two orientable topological spaces, with most studies focusing on the smooth structure of these spaces. In \cite{o53}, the concept of degree maps is used to address the homotopy classification problem for maps of a manifold $M^n$ into a manifold $Q^n$ where $\pi_r(Q^n) = 0$ for $1 < r < n$. Epstein \cite{Epstein} considered a notion of absolute degree from a perspective of the number of inverse images of individual points and their small neighborhoods under a proper map compatible with boundaries. A remarkable work on simplicial maps from an orientable $n$-pseudomanifold into $\mathbb{S}^m$ with the octahedral triangulation can be found in \cite{f67}. Several technical results regarding degree $d$ simplicial maps in dimensions two and three are discussed in \cite{B79}. It follows from the works \cite{Gromov1982, R24} that there does not exist a degree $d$ self-map of an orientable surface of genus greater than one for $d > 1$. For $n \geq 3$, explicit constructions of simplicial self-maps of degree $d$ for $n$-dimensional manifolds other than spheres are not known for all $d$ in the literature.

In this article, we focus on the degree of maps in the PL (piecewise-linear) category. Specifically, considering the simplicial cell complex structures, we study the concept of the degree of maps between two closed orientable PL $n$-manifolds. Our aim is to construct minimal simplicial degree $d$ self-maps of $\mathbb{S}^n$ and $\mathbb{S}^{n-1}\times \mathbb{S}^1$ for all integers $d$. A central motivation for studying minimal simplicial self-maps of $\mathbb{S}^{n-1}\times \mathbb{S}^1$ arises from the broader challenge of understanding the topological and combinatorial constraints on maps between manifolds in the PL category. 
While simplicial maps of prescribed degree between spheres have been extensively studied, comparatively little is known about such maps for manifolds like $\mathbb{S}^{n-1}\times \mathbb{S}^1$, which exhibit richer topological structures and additional subtleties. Our goal is to investigate whether such maps can be constructed explicitly and, if so, to determine minimal simplicial degree $d$ self-maps of $\mathbb{S}^{n-1}\times \mathbb{S}^1$ for all $d$ under simplicial constraints.
This work builds on classical results and aims to extend the theory of simplicial degree maps beyond the well-understood case of spheres, thereby shedding light on fundamental questions in PL topology and providing constructive methods for realizing degree $d$ maps in higher dimensions.

 The outline of this article is as follows: Before moving on to the main result (cf. Theorem \ref{theorem:main2}), involving self-maps of $\mathbb{S}^{n-1}\times \mathbb{S}^1$, of this article, we prove an easy result concerning the construction of a minimal simplicial degree $d$ map from a closed orientable PL $ n $-manifold to $ \mathbb{S}^n $, where $ n \geq 1 $  (cf. Theorem \ref{theorem:main1}). Specifically, this involves creating a degree $ d $ simplicial map from a $ (2\max\{|d|,1\}) $-facet colored triangulation of $ \mathbb{S}^n $ to the standard $ 2 $-facet colored triangulation of $ \mathbb{S}^n $. These triangulations are shown to be the minimal possible for a degree $ d $ simplicial self-map of $ \mathbb{S}^n $  (cf. Corollaries \ref{corollary:main} and \ref{corollary:main2}).

Next, we construct a $ (2(n+1)\max\{|d|,1\} )$-facet colored triangulation of $ \mathbb{S}^{n-1} \times I $. By identifying boundary components of this colored triangulation of $ \mathbb{S}^{n-1} \times I $ appropriately, we then construct a  $ (2(n+1)\max\{|d|,1\} )$-facet  colored triangulation of $ \mathbb{S}^{n-1} \times \mathbb{S}^1 $. On the other hand, we consider the standard $ 2(n+1) $-facet colored triangulation of $ \mathbb{S}^{n-1} \times \mathbb{S}^1 $ (cf. \cite{b19, gv87}). The standard colored triangulation of \( \mathbb{S}^{n-1} \times \mathbb{S}^1 \) was constructed in \cite{gv87}. To the best of our knowledge, there is no colored triangulation of \( \mathbb{S}^{n-1} \times \mathbb{S}^1 \) with fewer facets. For each $ d \in \mathbb{Z} $, we proceed to construct a degree $ d $ simplicial map from a $ (2(n+1)\max\{|d|,1\} )$-facet  colored triangulation of $ \mathbb{S}^{n-1} \times \mathbb{S}^1 $ to the standard $ 2(n+1)$-facet colored triangulation of $ \mathbb{S}^{n-1} \times \mathbb{S}^1 $. We prove that this $(2(n+1)\max\{|d|,1\} )$-facet  colored triangulation of $ \mathbb{S}^{n-1} \times \mathbb{S}^1 $ is the minimal possible colored triangulations for a degree $ d $ simplicial self-map of $ \mathbb{S}^{n-1} \times \mathbb{S}^1 $ with respect to the standard $2(n+1)$-facet colored triangulation of $ \mathbb{S}^{n-1} \times \mathbb{S}^1 $, where $ n \geq 2 $ (cf. Theorem \ref{theorem:main2} and Corollary \ref{corollary:main2}).

\section{Preliminaries}

All spaces and maps in this article are considered in the PL-category \cite{rs72}. Suppose that $K$ is a finite collection of closed balls, and let its geometric realization be  $|K| = \bigcup_{B\in K} B $. Then $K$ is called a \emph{simplicial cell complex} if the following conditions hold:

\begin{enumerate}[$(i)$]
\item $|K|=$ $\bigsqcup_{B\in K}$ int$(B)$, that is, $|K|$ is the disjoint union of the interiors of the balls in $K$;

\item for any $A,B \in K$, the intersection $A\cap B$ is a union of balls belonging to $K$;

\item  for each $h$-ball $A\in K$, the poset $\{B\in K \, | \, B \subset A\}$, ordered by inclusion, is isomorphic to the lattice of all faces of the standard $h$-simplex.
\end{enumerate}
\noindent A {\it pseudo-triangulation} of a polyhedron $P$ is a pair $(K,f)$, where $K$ is a simplicial cell complex and $f: |K| \to P$ is a PL-homeomorphism (see \cite{fgg86} for more details). A maximal closed ball of $K$ is called a \emph{facet}.  If all facets of $K$ have the same dimension $n$, then $K$ is called a 
\emph{pure $n$-dimensional simplicial cell complex}.

\begin{definition}\label{def:colored triangulation}
{\rm 
A pure $n$-dimensional simplicial cell complex $K$ is called \emph{$(n+1)$-colorable} if the vertices of its $1$-skeleton can be properly colored using $n+1$ colors. Let $M$ be a closed connected PL $n$-manifold. A simplicial cell complex $K$ is called a \emph{colored triangulation} of $M$ if $K$ is $(n+1)$-colorable and its geometric realization $|K|$ is PL-homeomorphic to $M$.}
\end{definition}

\begin{definition}\label{def:simplicial maps}
{\rm 
Let $K$ and $L$ be simplicial cell complexes. A function $f: K \to L$ is called a {\em simplicial map} if for each $A\in K$, $f(A)\in L$ and the poset $\{f(B) | \, B \subset A\}$, ordered by inclusion, is same as the poset $\{C\in  L\, | \, C \subset f(A)\}$.
}
\end{definition}
If $f$ is bijective, then its inverse $f^{-1}$ is a simplicial map from $L$ into $K$ and $f$ is called an {\em isomorphism}. For a simplicial map  $f: K \to L$, one can extend it to a continuous map $|f|: |K|\to |L|$ canonically.

\subsection{Crystallization} \label{crystal}
The crystallization theory provides a tool for representing piecewise-linear (PL) manifolds of any dimension combinatorially, using edge-colored graphs. Throughout the article, by a graph, we mean a multigraph with no loops. Let $\Gamma = (V(\Gamma), E(\Gamma))$  be an edge-colored multigraph with no loops, where the edges are colored (or labeled) using $\Delta_n := \{0, 1, \dots, n\}$. The elements of the set $\Delta_n$ are referred to as the {\it colors} of $\Gamma$. The coloring of $\Gamma$ is called a \textit{proper edge-coloring} if any two adjacent edges in $\Gamma$ have different labels. In other words, for a proper edge-coloring, there exists a map $\gamma: E(\Gamma) \to \Delta_n$ such that  $\gamma(e_1) \ne \gamma(e_2)$ for any two adjacent edges $e_1$ and $e_2$. We denote a properly edge-colored graph as $(\Gamma,\gamma)$, or simply as $\Gamma$ if the coloring is understood. If a graph $\Gamma$ is such that the degree of each vertex in the graph is $n+1$, then it is said to be {\it $(n+1)$-regular}.
We refer to \cite{bm08} for standard terminologies on graphs. 

An {\it $(n+1)$-regular colored graph} is a pair $(\Gamma,\gamma)$, where $\Gamma$ is $(n+1)$-regular and $\gamma$ is a proper edge-coloring of $\Gamma$. For each $\mathcal{C} \subseteq \Delta_n$ with cardinality $k$, the graph $\Gamma_\mathcal{C} = (V(\Gamma), \gamma^{-1}(\mathcal{C}))$ is a $k$-regular colored graph with edge-coloring $\gamma|_{\gamma^{-1}(\mathcal{C})}$. For a color set $\{j_1,j_2,\dots,j_k\} \subset \Delta_n$, $g(\Gamma_{\{j_1,j_2, \dots, j_k\}})$ or $g_{\{j_1, j_2, \dots, j_k\}}$ denotes the number of connected components of the graph $\Gamma_{\{j_1, j_2, \dots, j_k\}}$. Denote the set $\Delta_n\setminus \{j\}$ by $\hat{j}$. A graph $(\Gamma,\gamma)$ is called {\it contracted} if the subgraph $\Gamma_{\hat{j}} = \Gamma_{\Delta_n\setminus \{j\}}$ is connected for all $j \in \Delta_n$.
 
For a properly edge-colored graph $(\Gamma,\gamma)$ with the color set $\Delta_n$, a corresponding $(n+1)$-colorable simplicial cell complex ${\mathcal K}(\Gamma)$ is constructed as follows:
\begin{itemize}
\item{} For each vertex $v\in V(\Gamma)$, take an $n$-simplex $\sigma(v)$ with vertices labeled by $\Delta_n$.

\item{} Corresponding to each edge of color $j$ between $v_1,v_2\in V(\Gamma)$, identify the ($n-1$)-faces of $\sigma(v_1)$ and $\sigma(v_2)$ opposite to the $j$-labeled vertices such that the vertices with the same labels coincide.
\end{itemize}

\noindent Note that all colors of \(\Delta_n\) need not appear in the edge-coloring of \((\Gamma,\gamma)\). The topological space \( |\mathcal{K}(\Gamma)| \) inherits a natural PL structure, and the graph \((\Gamma,\gamma)\) is said to {\it represent} \( |\mathcal{K}(\Gamma)| \). If $(\G,\gamma)$ is an $(n+1)$-regular colored graph and $|\mathcal{K}(\Gamma)|$ is PL-homeomorphic to a closed $n$-manifold $M$, then $(\Gamma,\gamma)$ is called a \textit{gem} (graph encoded manifold) of $M$. Note that, by construction, $\mathcal{K}(\Gamma)$ is a colored triangulation of $M$. The {\it disjoint star} of \(\sigma \in \mathcal{K}(\Gamma)\) is a simplicial cell complex that consists of all the \(n\)-simplices of \(\mathcal{K}(\Gamma)\) that contain \(\sigma\), with re-identification of only their \((n-1)\)-faces containing \(\sigma\) as in \(\mathcal{K}(\Gamma)\). The {\it disjoint link} of \(\sigma \in \mathcal{K}(\Gamma)\) is the subcomplex of its disjoint star generated by the simplices that do not intersect $\sigma$. 

\begin{proposition}\label{prop:one-one}
Let $\Gamma$ be a properly edge-colored graph with color set $\Delta_n$, and let $\mathcal{C}\subseteq \Delta_n$ with cardinality $k+1$. Then, the disjoint star of each $k$-simplex whose vertices are labeled by the elements of $\mathcal{C}$ corresponds to a connected component of the subgraph $\Gamma_{\Delta_n\setminus \mathcal C}$ induced by the colors in $\Delta_n\setminus \mathcal{C}$. In particular, for every subset $\mathcal{C} \subset \Delta_n$ with cardinality $k+1$, $\mathcal{K}(\Gamma)$ has as many $ k $-simplices with vertices labeled by $\mathcal{C}$ as there are connected components of $\Gamma_{\Delta_n \setminus \mathcal{C}}$.
\end{proposition}

\begin{proof}
Let $\sigma$ be a $k$-simplex of $\mathcal{K}(\Gamma)$ whose vertices are labeled by the elements of $\mathcal{C}$. By the construction of $\mathcal{K}(\Gamma)$, each $n$-simplex containing $\sigma$ corresponds to a vertex of $\Gamma$. Furthermore, two such $n$-simplices are identified along an $(n-1)$-face containing $\sigma$ if and only if the corresponding vertices of $\Gamma$ are joined by an edge whose color belongs to $\Delta_n \setminus \mathcal{C}$.

More precisely, if $c \in \mathcal{C}$, then the $(n-1)$-face opposite to the vertex labeled by $c$ does not contain $\sigma$, and hence the corresponding identification is not present in the disjoint star of $\sigma$. On the other hand, if $c \in \Delta_n \setminus \mathcal{C}$, then the $(n-1)$-face opposite to the vertex labeled by $c$ contains $\sigma$, and therefore the corresponding identification is retained in the disjoint star of $\sigma$. Hence, the adjacency relations among the $n$-simplices in the disjoint star of $\sigma$ are completely determined by the edges colored by the elements of $\Delta_n \setminus \mathcal{C}$. 

Moreover, any two $n$-simplices in the disjoint star can be joined by a sequence of face identifications containing $\sigma$, and therefore the corresponding vertices in the dual graph are connected by a path consisting only of edges colored by the elements of $\Delta_n \setminus \mathcal{C}$. Thus, the dual graph associated with the disjoint star of $\sigma$ is a connected component of $\Gamma_{\Delta_n \setminus \mathcal{C}}$. This completes the proof.
\end{proof}

 An $(n+1)$-regular colored graph is a gem if and only if each component of $\G_{\hat{i}}$ is a gem of $\mathbb{S}^{n-1}$ for all $i\in \D_n$ \cite{fgg86}. In particular, any $3$-regular colored graph is a gem. For further information on CW complexes and related concepts, refer to \cite{bj84}. An $(n+1)$-regular colored gem $(\Gamma,\gamma)$ of a closed manifold $M$ is called a {\em crystallization} of $M$ if it is contracted. In other words, the corresponding simplicial cell complex $ \mathcal{K}(\Gamma)$ has exactly $n+1$ vertices.

\begin{proposition}[\cite{fgg86}, Theorem $4$]\label{bipartite}
 Let $(\Gamma,\gamma)$ be a crystallization of an 
n-manifold $M$. Then M is orientable if and only if $\G$ is bipartite.
    
\end{proposition}

A graph with a proper edge-coloring using $n+1$ colors is called an {\it $(n+1)$-colored graph with boundary} if it is $n$-regular with respect to the color set $\{0,1,\dots,n-1\}$ and not an $(n+1)$-regular graph. The vertex with degree $n+1$ is called an {\it internal} vertex, and the vertex with degree $n$ is called a {\it boundary} vertex. For each  $(n+1)$-colored graph $(\Gamma,\gamma)$ with boundary,  we define its  boundary graph $(\partial \Gamma,\partial \gamma)$ as follows:

\begin{itemize}
\item{} There is a bijection between $V(\partial \Gamma)$ and the set of boundary vertices of $\Gamma$.

\item{} Two vertices $u_1,u_2 \in V(\partial \Gamma)$ are joined in $\partial \Gamma$ by an edge of color $j$ if and only if $u_1$ and $u_2$ are joined in $\Gamma$  by a path consisting of edges alternately colored $j$ and $n$.
\end{itemize} 
It is known that, every compact PL $n$-manifold $M$ admits a gem, i.e., an $(n+1)$-colored graph (either an $(n+1)$-regular colored graph or an $(n+1)$-colored graph with boundary) that represents $M$ (see \cite{bb21, fgg86} for more details).
For the study of topological properties of low-dimensional manifolds via colored graphs, see \cite{cc21, cm16}. Additionally, the notion of colored graphs for singular manifolds was introduced in \cite{ccg18}.

Let $(\G,\gamma)$ be an $(n+1)$-regular colored graph representing a closed $n$-manifold $M$. Suppose $u$ and $v$ are two vertices of $\G$ such that they are joined by $h$ edges of colors from $\{i_1,i_2,\dots i_h\}\subset \D_n$, and the vertices $u,v$ lie in different components of $\G_{\D_n\setminus \{i_1,i_2,\dots, i_h\}}$, where $1\le h \le n$. Then $u$ and $v$ are said to form a {\it dipole of type $h$}, where $1\le h \le n$. Let $\G^\prime$ be an $(n+1)$-regular colored graph with $V(\G^\prime)=V(\G)\setminus \{u,v\}$ obtained from $\G$ by canceling this $h$-dipole. If $p,q\in V(\G^\prime)$ are such that $p$ is connected to $u$ and $q$ is connected to $v$ by edges of color $j\in \D_n\setminus \{i_1,i_2,\dots, i_h\}$ in $\G$, then $p$ and $q$ are connected by an edge of color $j$ in $\G^\prime$. All other edges of $\G$ that are neither incident with $u$ nor with $v$ remain the same in $\G^\prime$ as well. The process of obtaining $\G^\prime$ from $\G$ is called the {\it cancellation of an $h$-dipole}. The reverse of this process is called the {\it addition of an $h$-dipole}. Dipoles of type $1$ and $n$ are called {\it degenerate} dipoles, and dipoles of type $h$, $1< h < n$, are called {\it non-degenerate} dipoles. 

\begin{proposition}[\cite{fgg86}, Corollary $2$]\label{dipole}
 Two $(n + 1)$-regular colored graphs represent homeomorphic manifolds 
if and only if they can be obtained from each other by addings and/or cancellings of dipoles. 
    
\end{proposition}

\begin{proposition}\label{gem_bipartite}
Let $(\Gamma,\gamma)$ be a gem of a closed $n$-manifold $M$. Then $\Gamma$ has an even number of vertices. Moreover, $M$ is orientable if and only if $\Gamma$ is bipartite.   
\end{proposition}

\begin{proof}
 Since $\G$ is an $(n+1)$-regular colored graph with the proper edge-coloring $\gamma$, the subgraph $\Gamma_{\{i\}}$ consists precisely of $i$-colored edges, for $i\in \D_n$. Suppose that $\Gamma_{\{i\}}$ contains $p$ edges. Then the gem $\G$ has $2p$ vertices. 

 Now, let $\G^\prime$ be a contracted $(n+1)$-regular colored graph obtained after cancelling all the $1$-dipoles of $\G$. By Proposition \ref{dipole}, $\G^\prime$ is a crystallization of $M$. Since the addition or cancellation of a $1$-dipole preserves the bipartiteness of the graph, $\G^\prime$ is bipartite if and only if $\G$ is bipartite.  By Proposition \ref{bipartite}, $M$ is orientable if and only if $\G^\prime$ is orientable. Hence, $M$ is orientable if and only if $\G$ is orientable. 
\end{proof}

Let $(\G,\gamma)$ be an $(n+1)$-regular colored graph representing a closed manifold $M$. Suppose $\Lambda_1\subset V(\G)$ and $\Lambda_2 \subset V(\G)$ be such that the subgraphs $A_1$ and $A_2$ generated by $\Lambda_1$ and $\Lambda_2$, respectively, represent $n$-dimensional balls. Let there be an isomorphism $\Phi:A_1 \to A_2$ such that $u$ and $\Phi(u)$ are joined by an edge of color $i$ for each $u\in \Lambda_1$, and $\Lambda_1$ and  $\Lambda_2$ lie in different components of $\G_{\hat{i}}$, where $i\in \Delta_n$ is fixed.
Consider a new  $(n+1)$-regular colored graph  $\G'$ obtained from $\G$ as follows. Let $V(\G')=V(\G)\setminus (\Lambda_1 \cup \Lambda_2)$. For two vertices $p$ and $q$ in $V(\G')$, if $p$ is connected to $u$ and $q$ is connected to $\Phi(u)$ by edges of color $j\in \D_n \setminus  \{i\}$ in $\G$ where $u\in \Lambda_1$, then $p$ and $q$ are joined by an edge of color $j$ in $\G'$. On the other hand, if $p$ and $q$ are joined by an edge of color $j\in \D_n$ in $\G$, then $p$ and $q$ are joined by an edge of color $j$ in $\G'$. The process  to obtain  $\G'$  from $\G$ is called a {\it polyhedral glue move} with respect to $(\Phi,\Lambda_1,\Lambda_2,i)$. From \cite[Section $4$]{fg82i}, it is known that $\G'$ also represents $M$. If $\Lambda_1$ and $ \Lambda_2$ are singleton sets, then this polyhedral glue move is called {\it simple glue move} or {\it cancellation of $1$-dipole}, where $\Lambda_1$ and $\Lambda_2$ forms $1$-dipole with respect to the color $i$. For more details, one can see \cite{fg82i}.

\subsection{Regular Genus of closed PL $n$-manifolds}\label{sec:genus}
For a closed connected surface, its regular genus is simply its genus. However, for
closed connected PL $n$-manifolds ($n \geq 3$), the regular genus is defined as follows.
From \cite{g81}, \cite[Section $4$]{ ga81i}, it is known that if $(\Gamma,\gamma)$ is a bipartite (resp. non-bipartite) $(n+1)$-regular colored graph which represents a closed connected orientable (resp. non-orientable) PL $n$-manifold $M$, then for each permutation $\varepsilon=(\varepsilon_0,\dots,\varepsilon_n)$ of $\Delta_n$, there exists a regular embedding of $\Gamma$ into an orientable (resp. non-orientable) surface $S$. A {\it regular embedding} is an embedding where each region is bounded by a bi-colored cycle with colors $\varepsilon_i,\varepsilon_{i+1}$ for some $i$ (addition is modulo $n + 1$). Moreover, the Euler characteristic $\chi_\varepsilon(\Gamma)$ of the orientable (resp. non-orientable) surface  $S$ satisfies
$$\chi_\varepsilon(\Gamma)=\sum_{i \in \mathbb{Z}_{n+1}}g_{\{\varepsilon_i,\varepsilon_{i+1}\}} + (1-n)\frac{\text{card}(V(\Gamma))}{2},$$ 
and the genus (resp. half of genus) $\rho_ \varepsilon$ of $S$ satisfies
$$\rho_ \varepsilon(\Gamma)=1-\frac{\chi_\varepsilon(\Gamma)}{2}.$$
The regular genus $\rho(\Gamma)$ of $(\Gamma,\gamma)$ is defined as
$$\rho(\Gamma)= \min \{\rho_{\varepsilon}(\Gamma) \ | \  \varepsilon \ \mbox{ is a permutation of } \ \Delta_n\}.$$
The regular genus of $M$ is defined as 
$$\mathcal G(M) = \min \{\rho(\Gamma) \ | \  (\Gamma,\gamma) \mbox{ represents } M\}.$$
	
The regular genus is a PL invariant. A closed manifold of dimension $n$ with regular genus $0$ is characterized as $\S$ \cite{fg82}. Some recent works on the regular genus can be found in the following articles \cite{b19, bb21, bc17}.	

\begin{figure}[ht]
\tikzstyle{vert}=[circle, draw, fill=black!100, inner sep=0pt, minimum width=3pt]
\tikzstyle{ver}=[]
\tikzstyle{extra}=[circle, draw, fill=black!50, inner sep=0pt, minimum width=4pt]
\tikzstyle{edge} = [draw,thick,-]
\centering
\begin{tikzpicture}[scale=.8]

\begin{scope}[shift={(-6,0)}]
\foreach \x/\y/\z in {-1/5/0,1/5/1,-1/3/2,1/3/3,-1/1/4,1/1/5,-1/-1/6,1/-1/7,-1/-3/8,1/-3/9,-1/-5/10,1/-5/11}
{\node[vert] (a\z) at (\x,\y){};}
\foreach \x/\y in {-1/-1.5,-1/-2,-1/-2.5,1/-1.5,1/-2,1/-2.5}{\node[extra] () at (\x,\y){};}
\foreach \x/\y in {a0/a2,a2/a4,a4/a6,a8/a10,a1/a3,a3/a5,a5/a7,a9/a11}
{\draw [edge]  (\x) -- (\y);}

\foreach \x/\y/\z in {a0/a1/4.5,a0/a1/5.5,a2/a3/2.5,a2/a3/3.5,a4/a5/.5,a4/a5/1.5,a6/a7/-.5,a6/a7/-1.5,a8/a9/-2.5,a8/a9/-3.5,a10/a11/-4.5,a10/a11/-5.5}{
\draw[edge]  plot [smooth,tension=1.5] coordinates{(\x)(0,\z)(\y)};}

\node[ver] () at (-1.5,5.25){\small{$ v_1^1$}};
\node[ver] () at (1.5,5.25){\small{$ v_2^1$}};
\node[ver] () at (-1.5,3.25){\small{$ v_1^2$}};
\node[ver] () at (1.5,3.25){\small{$ v_2^2$}};
\node[ver] () at (-1.5,1.25){\small{$ v_1^3$}};
\node[ver] () at (1.5,1.25){\small{$ v_2^3$}};
\node[ver] () at (-1.5,-0.75){\small{$ v_1^4$}};
\node[ver] () at (1.5,-0.75){\small{$ v_2^4$}};
\node[ver] () at (-1.5,-2.75){\small{$ v_1^n$}};
\node[ver] () at (1.5,-2.75){\small{$ v_2^n$}};
\node[ver] () at (-1.7,-5.25){\small{$ v_1^{n+1}$}};
\node[ver] () at (1.7,-5.25){\small{$ v_2^{n+1}$}};

\node[ver] () at (0,5.1){\tiny{$0,1,\dots,$}};
\node[ver] () at (0,4.8){\tiny{$n-2$}};
\node[ver] () at (0,3.1){\tiny{$1,2,\dots,$}};
\node[ver] () at (0,2.8){\tiny{$n-2,n$}};
\node[ver] () at (0,1){\tiny{$2,3,\dots,n$}};
\node[ver] () at (0,-1){\tiny{$3,\dots,n,0$}};
\node[ver] () at (0,-2.8){\tiny{$n-1,n,$}};
\node[ver] () at (0,-3.1){\tiny{$0,\dots,n-4$}};
\node[ver] () at (0,-4.8){\tiny{$n-1,$}};
\node[ver] () at (0,-5.1){\tiny{$0,\dots,n-3$}};

\foreach \x/\y/\z in {-1.5/4/$n-1$,1.5/4/$n-1$,-1.3/2/0,1.3/2/0,-1.3/0/1,1.3/0/1,-1.5/-4/$n-2$,1.5/-4/$n-2$}{
\node[ver] () at (\x,\y){\tiny{\z}};}
\node[ver] () at (0,-6.45){$(a)$};
\end{scope}

\begin{scope}[shift={(0,0)}]
\foreach \x/\y/\z in {-1/5/0,1/5/1,-1/3/2,1/3/3,-1/1/4,1/1/5,-1/-1/6,1/-1/7,-1/-3/8,1/-3/9,-1/-5/10,1/-5/11}
{\node[vert] (a\z) at (\x,\y){};}
\foreach \x/\y in {-1/-1.5,-1/-2,-1/-2.5,1/-1.5,1/-2,1/-2.5}{\node[extra] () at (\x,\y){};}
\foreach \x/\y in {a0/a2,a2/a4,a4/a6,a8/a10,a1/a3,a3/a5,a5/a7,a9/a11}
{\draw [edge]  (\x) -- (\y);}

\foreach \x/\y/\z in {a0/a1/4.5,a0/a1/5.5,a2/a3/2.5,a2/a3/3.5,a4/a5/.5,a4/a5/1.5,a6/a7/-.5,a6/a7/-1.5,a8/a9/-2.5,a8/a9/-3.5,a10/a11/-4.5,a10/a11/-5.5}{
\draw[edge]  plot [smooth,tension=1.5] coordinates{(\x)(0,\z)(\y)};}

\node[ver] () at (-1.5,5.25){\small{$ v_1^1$}};
\node[ver] () at (1.5,5.25){\small{$ v_2^1$}};
\node[ver] () at (-1.5,3.25){\small{$ v_1^2$}};
\node[ver] () at (1.5,3.25){\small{$ v_2^2$}};
\node[ver] () at (-1.5,1.25){\small{$ v_1^3$}};
\node[ver] () at (1.5,1.25){\small{$ v_2^3$}};
\node[ver] () at (-1.5,-0.75){\small{$ v_1^4$}};
\node[ver] () at (1.5,-0.75){\small{$ v_2^4$}};
\node[ver] () at (-1.5,-2.75){\small{$ v_1^n$}};
\node[ver] () at (1.5,-2.75){\small{$ v_2^n$}};
\node[ver] () at (-1.7,-5.25){\small{$ v_1^{n+1}$}};
\node[ver] () at (1.7,-5.25){\small{$ v_2^{n+1}$}};

\node[ver] () at (0,5.1){\tiny{$0,1,\dots,$}};
\node[ver] () at (0,4.8){\tiny{$n-2$}};
\node[ver] () at (0,3.1){\tiny{$1,2,\dots,$}};
\node[ver] () at (0,2.8){\tiny{$n-2,n$}};
\node[ver] () at (0,1){\tiny{$2,3,\dots,n$}};
\node[ver] () at (0,-1){\tiny{$3,\dots,n,0$}};
\node[ver] () at (0,-2.8){\tiny{$n-1,n,$}};
\node[ver] () at (0,-3.1){\tiny{$0,\dots,n-4$}};
\node[ver] () at (0,-4.8){\tiny{$n-1,$}};
\node[ver] () at (0,-5.1){\tiny{$0,\dots,n-3$}};

\foreach \x/\y/\z in {-1.5/4/$n-1$,1.5/4/$n-1$,-1.3/2/0,1.3/2/0,-1.3/0/1,1.3/0/1,-1.5/-4/$n-2$,1.5/-4/$n-2$}{
\node[ver] () at (\x,\y){\tiny{\z}};}

\draw[edge] plot[smooth,tension=1] coordinates{(a0)(-2.5,2.5)(-2.5,-2.5)(a10)};

\node[ver] () at (-3,0){\tiny{$n$}};
\node[ver] () at (3,0){\tiny{$n$}};

\draw[edge] plot[smooth,tension=1] coordinates{(a1)(2.5,2.5)(2.5,-2.5)(a11)};

\node[ver] () at (0,-6.45){$(b)$};
\end{scope}

\begin{scope}[shift={(7,0.5)}]
\foreach \x/\y/\z in {-1/5/0,1/5/1,-1/3/2,1/3/3,-1/1/4,1/1/5,-1/-1/6,1/-1/7,-1/-3/8,1/-3/9,-1/-5/10,1/-5/11}
{\node[vert] (a\z) at (\x,\y){};}
\foreach \x/\y in {-1/-1.5,-1/-2,-1/-2.5,1/-1.5,1/-2,1/-2.5}{\node[extra] () at (\x,\y){};}
\foreach \x/\y in {a0/a2,a2/a4,a4/a6,a8/a10,a1/a3,a3/a5,a5/a7,a9/a11}
{\draw [edge]  (\x) -- (\y);}

\foreach \x/\y/\z in {a0/a1/4.5,a0/a1/5.5,a2/a3/2.5,a2/a3/3.5,a4/a5/.5,a4/a5/1.5,a6/a7/-.5,a6/a7/-1.5,a8/a9/-2.5,a8/a9/-3.5,a10/a11/-4.5,a10/a11/-5.5}{
\draw[edge]  plot [smooth,tension=1.5] coordinates{(\x)(0,\z)(\y)};}

\node[ver] () at (-1.5,5.25){\small{$ v_1^1$}};
\node[ver] () at (1.5,5.25){\small{$ v_2^1$}};
\node[ver] () at (-1.5,3.25){\small{$ v_1^2$}};
\node[ver] () at (1.5,3.25){\small{$ v_2^2$}};
\node[ver] () at (-1.5,1.25){\small{$ v_1^3$}};
\node[ver] () at (1.5,1.25){\small{$ v_2^3$}};
\node[ver] () at (-1.5,-0.75){\small{$ v_1^4$}};
\node[ver] () at (1.5,-0.75){\small{$ v_2^4$}};
\node[ver] () at (-1.5,-2.75){\small{$ v_1^n$}};
\node[ver] () at (1.5,-2.75){\small{$ v_2^n$}};
\node[ver] () at (-1.7,-5.25){\small{$ v_1^{n+1}$}};
\node[ver] () at (1.7,-5.25){\small{$ v_2^{n+1}$}};

\node[ver] () at (0,5.1){\tiny{$0,1,\dots,$}};
\node[ver] () at (0,4.8){\tiny{$n-2$}};
\node[ver] () at (0,3.1){\tiny{$1,2,\dots,$}};
\node[ver] () at (0,2.8){\tiny{$n-2,n$}};
\node[ver] () at (0,1){\tiny{$2,3,\dots,n$}};
\node[ver] () at (0,-1){\tiny{$3,\dots,n,0$}};
\node[ver] () at (0,-2.8){\tiny{$n-1,n,$}};
\node[ver] () at (0,-3.1){\tiny{$0,\dots,n-4$}};
\node[ver] () at (0,-4.8){\tiny{$n-1,$}};
\node[ver] () at (0,-5.1){\tiny{$0,\dots,n-3$}};

\foreach \x/\y/\z in {-1.5/4/$n-1$,1.5/4/$n-1$,-1.3/2/0,1.3/2/0,-1.3/0/1,1.3/0/1,-1.5/-4/$n-2$,1.5/-4/$n-2$}{
\node[ver] () at (\x,\y){\tiny{\z}};}

\draw[edge] plot[smooth,tension=1] coordinates{(a0)(-2.5,2.5)(-2,-6)(a11)};

\node[ver] () at (-3,0){\tiny{$n$}};
\node[ver] () at (3,0){\tiny{$n$}};

\draw[edge] plot[smooth,tension=1] coordinates{(a1)(2.5,2.5)(2,-6)(a10)};

\node[ver] () at (0,-7){$(c)$};

\end{scope}

\end{tikzpicture}
\caption{($a$) The standard crystallization of $\mathbb{S}^{n-1}\times I$, $n\geq 2$, ($b$) The standard crystallization of $\mathbb{S}^{n-1}\times \mathbb{S}^{1}$ when $n$ is odd, and ($c$) The standard crystallization of $\mathbb{S}^{n-1}\times \mathbb{S}^{1}$ when $n$ is even.}\label{fig:crys}
\end{figure}
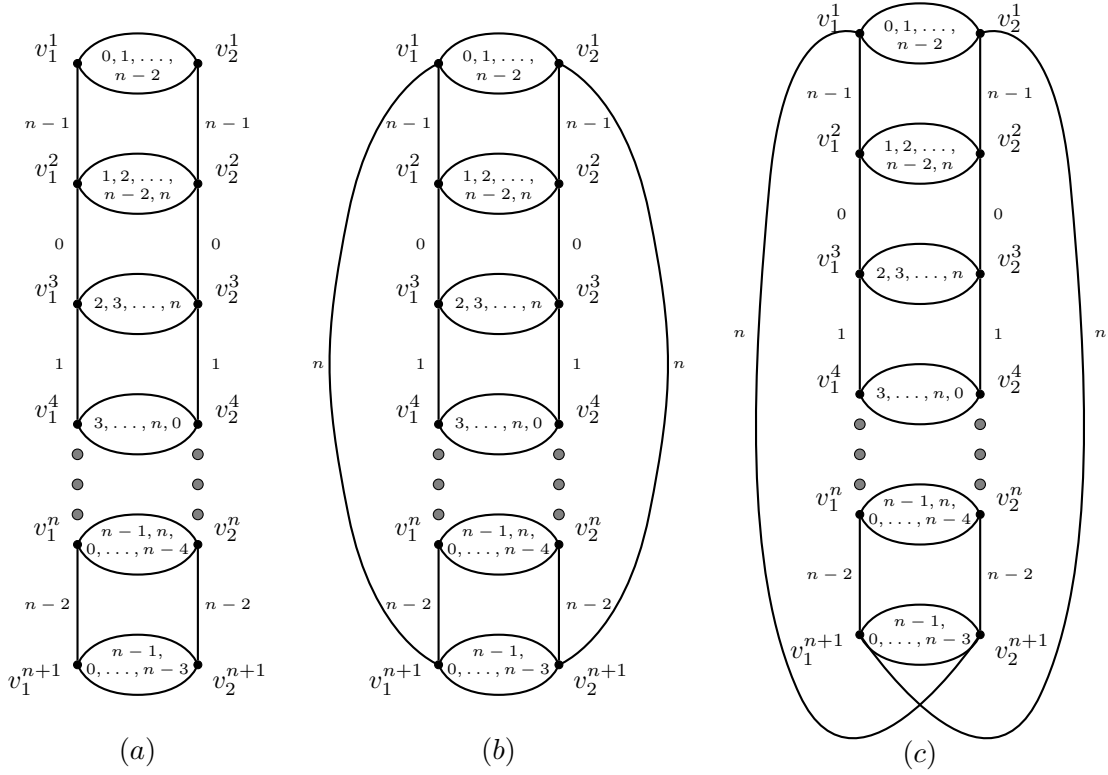

\section{Degree of simplicial maps}
Let \(\mathcal{K}\) be an oriented colored triangulation of a closed orientable \(n\)-manifold. 
Two \(k\)-simplices \(\sigma\) and \(\sigma'\) are said to be \emph{adjacent} if their 
intersection contains a \((k-1)\)-simplex. Since \(\mathcal{K}\) is oriented, its \(n\)-simplices 
are oriented so that any two adjacent $n$-simplices induce opposite orientations 
on their common \((n-1)\)-faces. Since the vertices of every \(n\)-simplex are 
colored by \(0, 1, \dots, n\), an orientation is determined by the ordering of the 
vertices according to their colors, up to even permutations (resp. up to odd permutations, 
giving the opposite orientation). Since \(\mathcal{K}\) is oriented, the orientations 
to its \(n\)-simplices satisfy the following condition: whenever an \(n\)-simplex is oriented by \([0, \dots, n]\) (resp. \(-[0, \dots, n]\)), every adjacent \(n\)-simplex must be oriented 
by \(-[0, \dots, n]\) (resp. \([0, \dots, n]\)).

An \(n\)-simplex with orientation \([0, \dots, n]\) is called \emph{positive}, 
while one with orientation \(-[0, \dots, n]\) is called \emph{negative}. 
Propagating this rule, a simplex adjacent to a positive simplex is declared 
negative, and a simplex adjacent to a negative simplex is declared positive. 
Since \(\mathcal{K}\) is oriented and triangulates an orientable \(n\)-manifold, this assignment 
is consistent, and hence every \(n\)-simplex of \(\mathcal{K}\) is uniquely 
classified as either positive or negative.

Let $(\Gamma,\gamma)$ be the $(n+1)$-regular colored graph representing 
$\mathcal{K}$. The vertices of $\Gamma$ corresponding to positive simplices 
are called \emph{positive vertices}, and the remaining vertices are called 
\emph{negative vertices}.

\begin{lemma}
Let $(\Gamma,\gamma)$ be an $(n+1)$-regular colored graph with $2p$ vertices, representing a closed orientable $n$-manifold $M$. Then the corresponding colored triangulation $\mathcal{K}(\Gamma)$ of $M$ can be oriented. Moreover, the numbers of positive and negative vertices are both equal to $p$.
\end{lemma}

\begin{proof}
It follows from Proposition~\ref{gem_bipartite} that $\Gamma$ is bipartite and has an even number of vertices, say $2p$. Let $V^+$ and $V^-$ denote the bipartition of the vertex set. Since, for each $i \in \D_n$, the graph $\Gamma_{\{i\}}$ has exactly $p$ edges, it follows that both $V^+$ and $V^-$ contain exactly $p$ vertices.

Assign to each $n$-simplex of $\mathcal{K}(\Gamma)$ corresponding to a vertex in $V^+$ the orientation $[0,\dots,n]$, and to each $n$-simplex corresponding to a vertex in $V^-$ the orientation $-[0,\dots,n]$. Then any two adjacent simplices induce opposite orientations on their common $(n-1)$-faces. Therefore, the colored triangulation $\mathcal{K}(\Gamma)$ is oriented. Accordingly, vertices in $V^+$ are called positive and those in $V^-$ are called negative. Thus, the numbers of positive and negative vertices are both equal to $p$.
\end{proof}

Let $\sigma_1,\sigma_2,\dots,\sigma_k$ be $n$-simplices of $\mathcal{K}$. The \emph{algebraic number} of the set $\{\sigma_1,\sigma_2,\dots,\sigma_k\}$ is defined to be the sum $\sum_{i=1}^{k}\operatorname{sign}(\sigma_i)$, where $\operatorname{sign}(\sigma)$ is $1$ or $-1$ according as $\sigma$ is a positive or a negative simplex, respectively. Similarly, for vertices $v_1,v_2,\dots,v_k$ of $\Gamma$, the \emph{algebraic number} of $\{v_1,v_2,\dots,v_k\}$ is defined to be the sum $\sum_{i=1}^{k}\operatorname{sign}(v_i)$, where $\operatorname{sign}(v)$ is $1$ or $-1$ depending on whether $v$ is a positive vertex or a negative vertex, respectively.

It is known that the $n^{th}$ homology (with $\mathbb{Z}$-coefficients) of a closed orientable $n$-manifold is isomorphic to $\mathbb{Z}$. 

\begin{definition}{\rm
Let $\mathcal K_1$ and $\mathcal K_2$ be oriented colored triangulations of closed orientable $n$-manifolds $M_1$ and $M_2$, respectively, and let $f:\mathcal K_1 \to \mathcal K_2$ be a simplicial map. Then $f$ induces a map $f_{\ast}: H_n(\mathcal K_1, \mathbb{Z})\to H_n(\mathcal K_2, \mathbb{Z})$, given by $f_{\ast}(a^{\mathcal K_1})=d_{f}\cdot b^{\mathcal K_2}$, where $d_{f}$ is an integer, called the {\it degree} of the map $f$, and $a^{\mathcal K_1}$ and $b^{\mathcal K_2}$ denote the generators of the homology groups $H_n(\mathcal K_1,\mathbb{Z})\cong \mathbb{Z}$ and $H_n(\mathcal K_2,\mathbb{Z})\cong \mathbb{Z}$, respectively.}
\end{definition}
Observe that $ a^{\mathcal{K}_1} $ is represented by the sum of positive $n$-simplices minus the sum of negative $n$-simplices in $ \mathcal{K}_1 $, and $ b^{\mathcal{K}_2} $ is represented by the sum of positive $n$-simplices minus the sum of negative $n$-simplices in $ \mathcal{K}_2 $.

\begin{definition}\label{def:minimal map}
{\rm
Let $M_1$ and $M_2$ be two closed orientable $n$-manifolds, and let $\mathcal{K}_1$ and $\mathcal{K}_2$ be oriented colored triangulations of $M_1$ and $M_2$, respectively. 
A map $f:\mathcal{K}_1 \to \mathcal{K}_2$ is called a \emph{minimal simplicial degree $d$ map with respect to $\mathcal K_2$} if $f$ has degree $d$ and there does not exist colored triangulation $\mathcal{K}_1'$ of $M_1$, together with a simplicial map 
$f' : \mathcal{K}_1' \to \mathcal{K}_2$ of degree $d$ such that $\mathcal{K}_1'$ has fewer facets than $\mathcal{K}_1$.
In particular, if $M_1 = M_2 = M$, then $f$ is called a \emph{minimal simplicial degree $d$ self-map with respect to $\mathcal K_2$}.}
\end{definition}

\begin{lemma}\label{gmap}
Let $(\Gamma_1,\gamma_1)$ (resp.  $(\Gamma_2,\gamma_2)$) be the $(n+1)$-regular colored graph representing an oriented colored triangulation $\mathcal{K}_1$ (resp. $\mathcal{K}_2$)  of a closed orientable $n$-manifold $M_1$ (resp. $M_2$).
 Suppose $g: V(\G_1)\to V(\G_2)$ is a map such that whenever $u$ and $v$ are joined by an edge of color $i\in \D_n$ in $\G_1$, the vertices $g(u)$ and $g(v)$ are either equal or joined by an edge of color $i$ in $\G_2$. Then $g$ induces a simplicial map $f:\mathcal K_1\to \mathcal K_2$.
\end{lemma}
\begin{proof}
For each $u\in \G_1$, let $\sigma_u$, $\sigma_{g(u)}$ be the $n$-simplices 
corresponding to $u$ and $g(u)$ in $\mathcal K_1$ and $\mathcal K_2$, respectively.
Define a map $f:\mathcal K_1 \to \mathcal K_2$ by sending the $n$-simplex $\sigma_u$ to $\sigma_{g(u)}=f(\sigma_u)$ such that 
each vertex of $\sigma_u$ is mapped to the vertex of
$\sigma_{g(u)}$ which has the same color. Note that the map $f$ preserves the dimension of a simplex in $\mathcal K_1$.

Let  $\sigma_u$ and $\sigma_v$ be the the $n$-simplices in $\mathcal K_1$
corresponding to $u$ and $v$ in $\G_1$, respectively. 
Suppose $\sigma_u$ and $\sigma_v$ share a common $k$-simplex $\tau$ labeled by the  color set $C\subset \D_n$ in $\mathcal K_1$. Then $\sigma_u$ and $\sigma_v$ belong to the disjoint star of the $k$-simplex $\tau$. Therefore, by Proposition \ref{prop:one-one}, $u$ and $v$ lie in the same component of the subgraph, induced by the color set $\Delta_n\setminus C$, of $\G_1$. By the definition of $g$, the vertices $g(u)$ and $g(v)$ lie in the same component of the subgraph, induced by the color set $\Delta_n\setminus C$, of $\G_2$. Hence, $\sigma_{g(u)}=f(\sigma_u)$ and $\sigma_{g(v)}=f(\sigma_v)$ share a common $k$-simplex 
$f(\tau)$ labeled by the color set $C$. Thus, $f$ is well defined. 

Moreover, by the construction of $f$, for each $A\in \mathcal K_1$, $f(A)\in \mathcal K_2$, and the poset 
$\{f(B)\mid B\subset A\}$, ordered by inclusion, coincides with the poset 
$\{C\in \mathcal K_2 \mid C\subset f(A)\}$. Therefore, $f$ is a simplicial map.
\end{proof}

If $g$ is not surjective, then the degree of $f$ is zero. If $g$ is surjective, then $d_f$ is the same as the product of $sign(v)$ and the algebraic number of $g^{-1}(v)$, where $v \in \G_2 $.

\section{Main Results}
In this section, we will construct minimal simplicial degree $d$ self-maps of $\mathbb{S}^{n}$ and $\mathbb{S}^{n-1}\times \mathbb{S}^1$, for $n \geq 2$ and $d \in \mathbb{Z}$.

\begin{figure}[ht]
\tikzstyle{vert}=[circle, draw, fill=black!100, inner sep=0pt, minimum width=3pt]
\tikzstyle{ver}=[]
\tikzstyle{extra}=[circle, draw, fill=black!50, inner sep=0pt, minimum width=4pt]
\tikzstyle{edge} = [draw,thick,-]
\centering
\begin{tikzpicture}[scale=.8]

\begin{scope}[shift={(0,0)}]
\foreach \x/\y/\z in
{-8/0/1,-5.5/0/2,-2.5/0/3,0/0/4,5/0/5,7.5/0/6}
{\node[vert] (\z) at (\x,\y){};}

\foreach \x/\y in
{1.5/0,2.5/0,3.5/0}
{\node[extra] at (\x,\y){};}

\foreach \x/\y/\z in
{-8.5/-0.5/v_1,-5/-0.5/v_2,-3/-0.5/v_3,0.5/-0.5/v_4,4.3/-0.5/v_{2d-1},8/-0.5/v_{2d}}
{\node[ver] () at (\x,\y){$\z$};}

\draw[edge] (2)--(3);

\foreach \x/\y/\z in 
{1/2/-6.8,3/4/-1.3,5/6/6.2}
{\draw[edge] plot [smooth,tension=1] coordinates{(\x)(\z,0.6)(\y)};}

\foreach \x/\y/\z in 
{1/2/-6.8,3/4/-1.3,5/6/6.2}
{\draw[edge] plot [smooth,tension=1] coordinates{(\x)(\z,-0.6)(\y)};}

\foreach \x in
{-6.8,-1.3,6.2}
{\node[ver] () at (\x,-0.3){\tiny{$n-1$}}; }

\foreach \x in
{-6.8,-1.3,6.2}
{\node[ver] () at (\x,0.1){\tiny{$0,1,\dots,$}}; }

\draw[edge] plot[smooth,tension=0.5] coordinates{(1) (-6.5,-1) (6,-1) (6)};

\node[ver] () at (0,-1.3){\tiny{$n$}};
\node[ver] () at (-3.9,0.2){\tiny{$n$}};

\end{scope}

\begin{scope}[shift={(0,-2.5)}]
\foreach \x/\y/\z in
{-1.5/0/1,1.5/0/2}
{\node[vert] (\z) at (\x,\y){};}

\foreach \x/\y/\z in
{-2/0/v^1,2/0/v^2}
{\node[ver] () at (\x,\y){$\z$};}

\draw[edge] plot[smooth,tension=1] coordinates{(1) (0,-0.6) (2)};

\draw[edge] plot[smooth,tension=1] coordinates{(1) (0,0.6) (2)};

\node[ver] () at (0,0){\tiny{$0,1,\dots,n$}};

\end{scope}

\end{tikzpicture}
\caption{A gem of $\mathbb{S}^{n}$ corresponding to degree $d$ and the standard crystallization of $\mathbb{S}^{n}$.}\label{fig:Sn}
\end{figure}
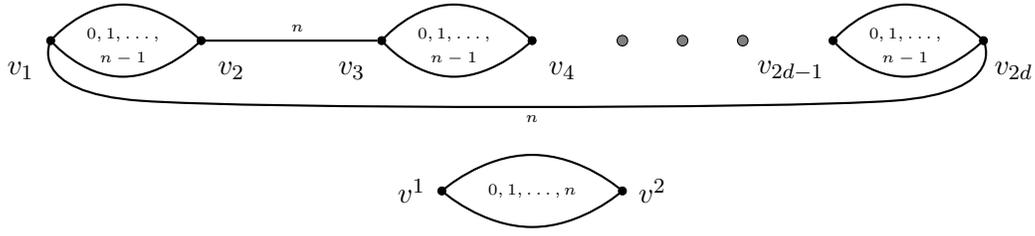

\begin{theorem} \label{theorem:main1}
Let $M$ be a closed orientable $n$-manifold, where $n \geq 1$. Then, for each $d \in \mathbb{Z}$, there exists a minimal simplicial degree $d$ map from $M$ to $\mathbb{S}^{n}$, with respect to the standard $2$-facet colored triangulation of $\mathbb{S}^{n}$. 
\end{theorem}

\begin{proof}
Let $(\G',\gamma')$ be the standard crystallization of $\mathbb{S}^{n}$ consisting of two vertices, $v^1$ and $v^2$ (Figure \ref{fig:Sn}), where $v^1$ is a positive vertex and $v^2$ is a negative vertex. Let $(\G,\gamma)$ be a crystallization of $M$ with the minimal number of vertices, $2p$. Since $ M $ is orientable, $ \Gamma $ is bipartite. Therefore, the vertices can be divided into two sets, $ A $ and $ B $. Let $ A $ consist of positive vertices and $ B $ consist of negative vertices. Let $ g: V(\Gamma) \to V(\Gamma') $ be a map where $ u \in V(\Gamma) $ is mapped to one of the two vertices of $ V(\Gamma') $, say $ v^1 $. If $ v \in V(\Gamma) $ is adjacent to $ u $, then $ g(v) $ can be either $ v^1 $ or $ v^2 $. Since $ \Gamma' $ consists of only two vertices, this implies that $ g(u) $ and $ g(v) $ are either the same, or an edge of color $ i $ is incident with $ g(u) $ and $ g(v) $ whenever $ u $ and $ v $ are joined by an edge of color $ i $ in $ \Gamma $. Therefore, a map from $V(\G)$ to $V(\G')$ induces a simplicial map from $M$ to $\mathbb{S}^{n}$ (cf. Lemma \ref{gmap}). If $g_0: V(\G)\to V(\G')$ is a constant function, then the degree of the simplicial map induced by $g_0$ is $0$, and it is minimal because $\G$ is minimal. Now, consider $g':V(\G')\to V(\G')$ such that $g'(v^1)=v^2$ and $g'(v^2)=v^1$. Clearly, the degree of the induced simplicial self-map of $\mathbb{S}^{n}$ is $-1$. Thus, if $g:\G \to \G'$ induces a simplicial map of degree $d$, then $g' \circ g$ induces a simplicial map of the degree $-d$. Clearly, $g$ induces a minimal simplicial map of degree $d$ if and only if $g' \circ g$ induces a minimal simplicial degree $-d$ map.

\noindent \textbf{Case $1$ ($ |d| \le p$):} Let $d$ be a positive integer. Define $g_d: V(\G)\to V(\G')$ such that any $d$ vertices of $B$ are assigned $v^2$, and all the remaining vertices of $\G$ are assigned $v^1$. Therefore, this map $g_d$ will induce the simplicial map $f:M \to \mathbb{S}^{n}$ of degree $d$. These maps are minimal as $\G$ is minimal. As pointed out above, $g_{-d}=g' \circ g_d$ will induce a simplicial map from $M$ to $\mathbb{S}^{n}$ of degree $-d$. \\
\noindent \textbf{Case $2$ ($|d| > p$):} Let $d$ be a positive integer. Adding $d-p$ number of $n$-dipoles to $\G$, we get a gem $\G_1$ of $M$ which consists of $d$ positive (resp. negative) vertices. Define $g_d: V(\G_1)\to V(\G')$ such that all the positive (resp. negative) vertices get mapped to $v^1$ (resp. $v^2$). Therefore, the map $g_d$ induces the simplicial map $f:M\to \mathbb{S}^{n}$ of degree $d$. These maps are minimal because, in the inverse image of a positive (resp. negative) simplex, there are exactly $d$ positive (resp. negative) simplices, and $\G'$ is minimal. Again, $g_{-d}=g' \circ g_d$ will induce a simplicial map from $M$ to $\mathbb{S}^{n}$ of degree $-d$.  
\end{proof}

\begin{figure}[ht]
\tikzstyle{vert}=[circle, draw, fill=black!100, inner sep=0pt, minimum width=3pt]
\tikzstyle{ver}=[]
\tikzstyle{extra}=[circle, draw, fill=black!50, inner sep=0pt, minimum width=4pt]
\tikzstyle{edge} = [draw,thick,-]
\centering
\begin{tikzpicture}[scale=.8]

\begin{scope}[shift={(0,0)}]
\foreach \x/\y/\z in {-1/5/0,1/5/1,-1/3/2,1/3/3,-1/1/4,1/1/5,-1/-1/6,1/-1/7,-1/-3/8,1/-3/9,-1/-5/10,1/-5/11}
{\node[vert] (a\z) at (\x,\y){};}
\foreach \x/\y in {-1/-1.5,-1/-2,-1/-2.5,1/-1.5,1/-2,1/-2.5}{\node[extra] () at (\x,\y){};}
\foreach \x/\y in {a0/a2,a2/a4,a4/a6,a8/a10,a1/a3,a3/a5,a5/a7,a9/a11}
{\draw [edge]  (\x) -- (\y);}

\foreach \x/\y/\z in {a0/a1/4.5,a0/a1/5.5,a2/a3/2.5,a2/a3/3.5,a4/a5/.5,a4/a5/1.5,a6/a7/-.5,a6/a7/-1.5,a8/a9/-2.5,a8/a9/-3.5,a10/a11/-4.5,a10/a11/-5.5}{
\draw[edge]  plot [smooth,tension=1.5] coordinates{(\x)(0,\z)(\y)};}

\node[ver] () at (-1.5,5.25){\small{$ v_1^1$}};
\node[ver] () at (1.5,5.25){\small{$ v_2^1$}};
\node[ver] () at (-1.5,3.25){\small{$ v_1^2$}};
\node[ver] () at (1.5,3.25){\small{$ v_2^2$}};
\node[ver] () at (-1.5,1.25){\small{$ v_1^3$}};
\node[ver] () at (1.5,1.25){\small{$ v_2^3$}};
\node[ver] () at (-1.5,-0.75){\small{$ v_1^4$}};
\node[ver] () at (1.5,-0.75){\small{$ v_2^4$}};
\node[ver] () at (-1.5,-2.75){\small{$ v_1^n$}};
\node[ver] () at (1.5,-2.75){\small{$ v_2^n$}};
\node[ver] () at (-1.7,-5.25){\small{$ v_1^{n+1}$}};
\node[ver] () at (1.7,-5.25){\small{$ v_2^{n+1}$}};

\node[ver] () at (0,5.1){\tiny{$0,1,\dots,$}};
\node[ver] () at (0,4.8){\tiny{$n-2$}};
\node[ver] () at (0,3.1){\tiny{$1,2,\dots,$}};
\node[ver] () at (0,2.8){\tiny{$n-2$}};
\node[ver] () at (0,1){\tiny{$2,\dots,n-1$}};
\node[ver] () at (0,-0.8){\tiny{$3,\dots,$}};
\node[ver] () at (0,-1.1){\tiny{$n-1,0$}};
\node[ver] () at (0,-2.8){\tiny{$n-1,$}};
\node[ver] () at (0,-3.1){\tiny{$0,\dots,n-4$}};
\node[ver] () at (0,-4.8){\tiny{$n-1,$}};
\node[ver] () at (0,-5.1){\tiny{$0,\dots,n-3$}};

\foreach \x/\y/\z in {-1.5/4/$n-1$,1.5/4/$n-1$,-1.3/2/0,1.3/2/0,-1.3/0/1,1.3/0/1,-1.5/-4/$n-2$,1.5/-4/$n-2$}{
\node[ver] () at (\x,\y){\tiny{\z}};}
\end{scope}

\begin{scope}[shift={(6,0)}]
\foreach \x/\y/\z in {-1/5/0,1/5/1,-1/3/2,1/3/3,-1/1/4,1/1/5,-1/-1/6,1/-1/7,-1/-3/8,1/-3/9,-1/-5/10,1/-5/11}
{\node[vert] (b\z) at (\x,\y){};}
\foreach \x/\y in {-1/-1.5,-1/-2,-1/-2.5,1/-1.5,1/-2,1/-2.5}{\node[extra] () at (\x,\y){};}
\foreach \x/\y in {b0/b2,b2/b4,b4/b6,b8/b10,b1/b3,b3/b5,b5/b7,b9/b11}
{\draw [edge]  (\x) -- (\y);}

\foreach \x/\y/\z in {b0/b1/4.5,b0/b1/5.5,b2/b3/2.5,b2/b3/3.5,b4/b5/.5,b4/b5/1.5,b6/b7/-.5,b6/b7/-1.5,b8/b9/-2.5,b8/b9/-3.5,b10/b11/-4.5,b10/b11/-5.5}{
\draw[edge]  plot [smooth,tension=1.5] coordinates{(\x)(0,\z)(\y)};}

\node[ver] () at (-1.5,5.25){\small{$ v_3^1$}};
\node[ver] () at (1.5,5.25){\small{$ v_4^1$}};
\node[ver] () at (-1.5,3.25){\small{$ v_3^2$}};
\node[ver] () at (1.5,3.25){\small{$ v_4^2$}};
\node[ver] () at (-1.5,1.25){\small{$ v_3^3$}};
\node[ver] () at (1.5,1.25){\small{$ v_4^3$}};
\node[ver] () at (-1.5,-0.75){\small{$ v_3^4$}};
\node[ver] () at (1.5,-0.75){\small{$ v_4^4$}};
\node[ver] () at (-1.5,-2.75){\small{$ v_3^n$}};
\node[ver] () at (1.5,-2.75){\small{$ v_4^n$}};
\node[ver] () at (-1.7,-5.25){\small{$ v_3^{n+1}$}};
\node[ver] () at (1.7,-5.25){\small{$ v_4^{n+1}$}};

\node[ver] () at (0,5.1){\tiny{$0,1,\dots,$}};
\node[ver] () at (0,4.8){\tiny{$n-2$}};
\node[ver] () at (0,3.1){\tiny{$1,2,\dots,$}};
\node[ver] () at (0,2.8){\tiny{$n-2$}};
\node[ver] () at (0,1){\tiny{$2,\dots,n-1$}};
\node[ver] () at (0,-0.8){\tiny{$3,\dots,$}};
\node[ver] () at (0,-1.1){\tiny{$n-1,0$}};
\node[ver] () at (0,-2.8){\tiny{$n-1,$}};
\node[ver] () at (0,-3.1){\tiny{$0,\dots,n-4$}};
\node[ver] () at (0,-4.8){\tiny{$n-1,$}};
\node[ver] () at (0,-5.1){\tiny{$0,\dots,n-3$}};

\foreach \x/\y/\z in {-1.5/4/$n-1$,1.5/4/$n-1$,-1.3/2/0,1.3/2/0,-1.3/0/1,1.3/0/1,-1.5/-4/$n-2$,1.5/-4/$n-2$}{
\node[ver] () at (\x,\y){\tiny{\z}};}
\end{scope}

\begin{scope}[shift={(13.5,0)}]
\foreach \x/\y/\z in {-1/5/0,1/5/1,-1/3/2,1/3/3,-1/1/4,1/1/5,-1/-1/6,1/-1/7,-1/-3/8,1/-3/9,-1/-5/10,1/-5/11}
{\node[vert] (c\z) at (\x,\y){};}
\foreach \x/\y in {-1/-1.5,-1/-2,-1/-2.5,1/-1.5,1/-2,1/-2.5}{\node[extra] () at (\x,\y){};}
\foreach \x/\y in {c0/c2,c2/c4,c4/c6,c8/c10,c1/c3,c3/c5,c5/c7,c9/c11}
{\draw [edge]  (\x) -- (\y);}

\foreach \x/\y/\z in {c0/c1/4.5,c0/c1/5.5,c2/c3/2.5,c2/c3/3.5,c4/c5/.5,c4/c5/1.5,c6/c7/-.5,c6/c7/-1.5,c8/c9/-2.5,c8/c9/-3.5,c10/c11/-4.5,c10/c11/-5.5}{
\draw[edge]  plot [smooth,tension=1.5] coordinates{(\x)(0,\z)(\y)};}

\node[ver] () at (-1.7,5.25){\small{$ v_{2d-1}^1$}};
\node[ver] () at (1.5,5.25){\small{$ v_{2d}^1$}};
\node[ver] () at (-1.7,3.25){\small{$ v_{2d-1}^2$}};
\node[ver] () at (1.5,3.25){\small{$ v_{2d}^2$}};
\node[ver] () at (-1.7,1.25){\small{$ v_{2d-1}^3$}};
\node[ver] () at (1.5,1.25){\small{$ v_{2d}^3$}};
\node[ver] () at (-1.7,-0.75){\small{$ v_{2d-1}^4$}};
\node[ver] () at (1.5,-0.75){\small{$ v_{2d}^4$}};
\node[ver] () at (-1.7,-2.75){\small{$ v_{2d-1}^n$}};
\node[ver] () at (1.6,-2.75){\small{$ v_{2d}^n$}};
\node[ver] () at (-1.7,-5.25){\small{$ v_{2d-1}^{n+1}$}};
\node[ver] () at (1.7,-5.25){\small{$ v_{2d}^{n+1}$}};

\node[ver] () at (0,5.1){\tiny{$0,1,\dots,$}};
\node[ver] () at (0,4.8){\tiny{$n-2$}};
\node[ver] () at (0,3.1){\tiny{$1,2,\dots,$}};
\node[ver] () at (0,2.8){\tiny{$n-2$}};
\node[ver] () at (0,1){\tiny{$2,\dots,n-1$}};
\node[ver] () at (0,-0.8){\tiny{$3,\dots,$}};
\node[ver] () at (0,-1.1){\tiny{$n-1,0$}};
\node[ver] () at (0,-2.8){\tiny{$n-1,$}};
\node[ver] () at (0,-3.1){\tiny{$0,\dots,n-4$}};
\node[ver] () at (0,-4.8){\tiny{$n-1,$}};
\node[ver] () at (0,-5.1){\tiny{$0,\dots,n-3$}};

\foreach \x/\y/\z in {-1.5/4/$n-1$,1.5/4/$n-1$,-1.3/2/0,1.3/2/0,-1.3/0/1,1.3/0/1,-1.5/-4/$n-2$,1.5/-4/$n-2$}{
\node[ver] () at (\x,\y){\tiny{\z}};}
\end{scope}

\foreach \x/\y in {8.5/3,9.5/3,10.5/3,8.5/1,9.5/1,10.5/1,8.5/-1,9.5/-1,10.5/-1,8.5/-3,9.5/-3,10.5/-3}
{\node[extra] () at (\x,\y){};}

\foreach \x/\y in 
{a3/b2,a5/b4,a7/b6,a9/b8}
{\draw [edge]  (\x) -- (\y);}

\foreach \x/\y in
{3/3.2,3/1.2,3/-0.8,3/-2.8}
{\node[ver] () at (\x,\y){\tiny{$n$}};}

\foreach \x/\y/\z in 
{a2/c3/2.2,a4/c5/0.2,a6/c7/-1.8,a8/c9/-3.8}
{\draw[edge] plot [smooth,tension=0.6] coordinates{(\x)(-1,\z) (7,\z) (14.5,\z)(\y)};}

\foreach \x in
{2.4,0.4,-1.6,-3.6}
{\node[ver] () at (10,\x){\tiny{$n$}};}

\end{tikzpicture}

\caption{A gem of $\mathbb{S}^{n-1}\times I$ with $2d(n+1)$ vertices.}\label{fig:I}
\end{figure}
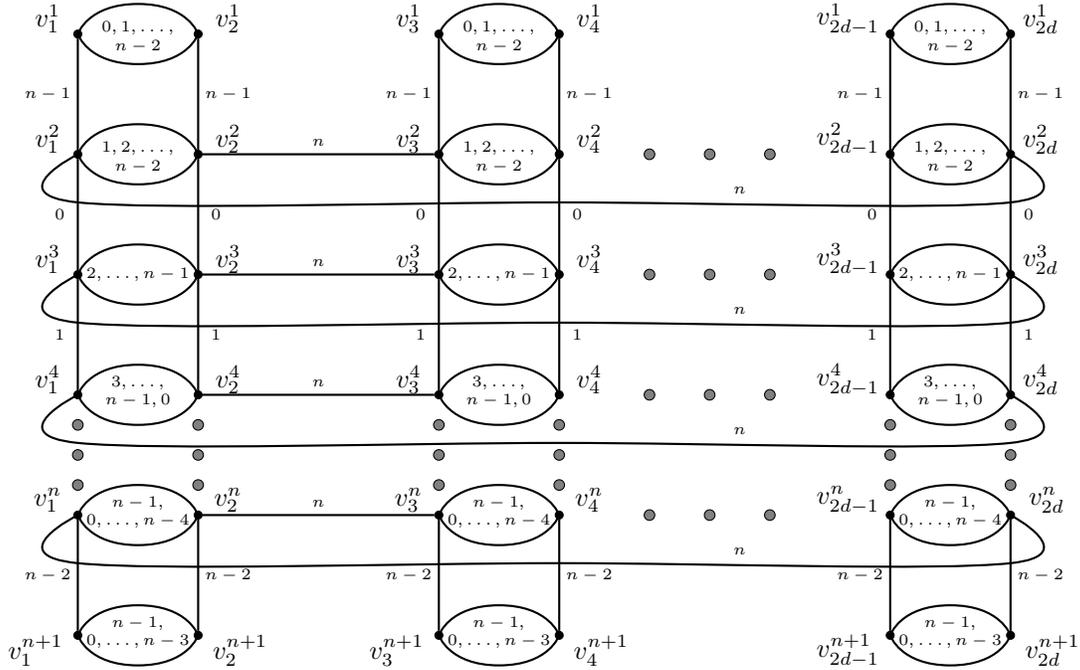

\begin{corollary}\label{corollary:main}
For each $d \in \mathbb{Z}$, there exists a minimal simplicial degree $d$ self-map of $\mathbb{S}^{n}$ with respect to the standard $2$-facet colored triangulation of $\mathbb{S}^{n}$, where $n\geq 1$. 
\end{corollary}

\begin{proof}
If we take $M$ to be $\mathbb{S}^{n}$ in Theorem \ref{theorem:main1}, then we get the desired result. Figure \ref{fig:Sn} exhibits a gem of $\mathbb{S}^{n}$ corresponding to degree $d,\ d\geq 1$. If we map $v_{2i-1}$, for $1\leq i\leq d$, to $v^1$ and others to $v^2$, then we get a minimal simplicial degree $d$ self-map of $\mathbb{S}^{n}$.   \end{proof}

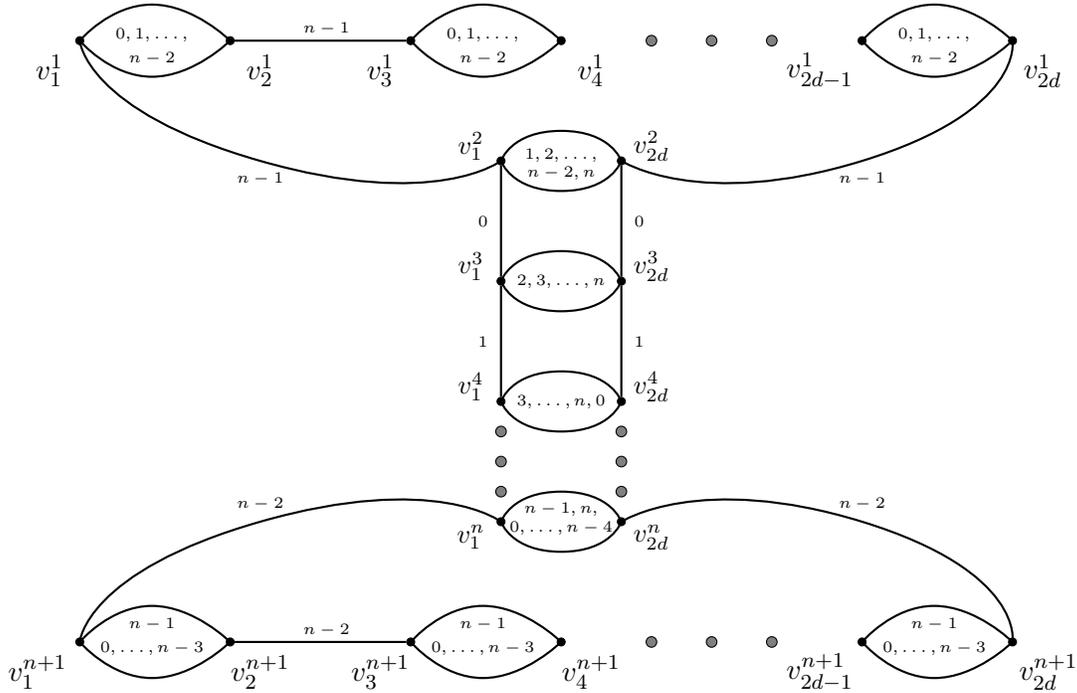
\begin{figure}[ht]
\tikzstyle{vert}=[circle, draw, fill=black!100, inner sep=0pt, minimum width=3pt]
\tikzstyle{ver}=[]
\tikzstyle{extra}=[circle, draw, fill=black!50, inner sep=0pt, minimum width=4pt]
\tikzstyle{edge} = [draw,thick,-]
\centering
\begin{tikzpicture}[scale=.8]

\begin{scope}[shift={(0,0)}]
\foreach \x/\y/\z in {-1/3/2,1/3/3,-1/1/4,1/1/5,-1/-1/6,1/-1/7,-1/-3/8,1/-3/9}
{\node[vert] (a\z) at (\x,\y){};}
\foreach \x/\y in {-1/-1.5,-1/-2,-1/-2.5,1/-1.5,1/-2,1/-2.5}{\node[extra] () at (\x,\y){};}
\foreach \x/\y in {a2/a4,a4/a6,a3/a5,a5/a7}
{\draw [edge]  (\x) -- (\y);}

\foreach \x/\y/\z in {a2/a3/2.5,a2/a3/3.5,a4/a5/.5,a4/a5/1.5,a6/a7/-.5,a6/a7/-1.5,a8/a9/-2.5,a8/a9/-3.5}{
\draw[edge]  plot [smooth,tension=1.5] coordinates{(\x)(0,\z)(\y)};}

\node[ver] () at (-1.5,3.25){\small{$ v_1^2$}};
\node[ver] () at (1.5,3.25){\small{$ v_{2d}^2$}};
\node[ver] () at (-1.5,1.25){\small{$ v_1^3$}};
\node[ver] () at (1.5,1.25){\small{$ v_{2d}^3$}};
\node[ver] () at (-1.5,-0.75){\small{$ v_1^4$}};
\node[ver] () at (1.5,-0.75){\small{$ v_{2d}^4$}};
\node[ver] () at (-1.5,-3.25){\small{$ v_1^n$}};
\node[ver] () at (1.5,-3.25){\small{$ v_{2d}^n$}};

\node[ver] () at (0,3.1){\tiny{$1,2,\dots,$}};
\node[ver] () at (0,2.8){\tiny{$n-2,n$}};
\node[ver] () at (0,1){\tiny{$2,3,\dots,n$}};
\node[ver] () at (0,-1){\tiny{$3,\dots,n,0$}};
\node[ver] () at (0,-2.8){\tiny{$n-1,n,$}};
\node[ver] () at (0,-3.1){\tiny{$0,\dots,n-4$}};

\foreach \x/\y/\z in {-1.3/2/0,1.3/2/0,-1.3/0/1,1.3/0/1}{
\node[ver] () at (\x,\y){\tiny{\z}};}
\end{scope}

\begin{scope}[shift={(0,5)}]
\foreach \x/\y/\z in
{-8/0/1,-5.5/0/2,-2.5/0/3,0/0/4,5/0/5,7.5/0/6}
{\node[vert] (b\z) at (\x,\y){};}

\foreach \x/\y in
{1.5/0,2.5/0,3.5/0}
{\node[extra] at (\x,\y){};}

\foreach \x/\y/\z in
{-8.5/-0.5/v_1^1,-5/-0.5/v_2^1,-3/-0.5/v_3^1,0.5/-0.5/v_4^1,4.3/-0.5/v_{2d-1}^1,8/-0.5/v_{2d}^1}
{\node[ver] () at (\x,\y){$\z$};}

\draw[edge] (b2)--(b3);

\foreach \x/\y/\z in 
{b1/b2/-6.8,b3/b4/-1.3,b5/b6/6.2}
{\draw[edge] plot [smooth,tension=1] coordinates{(\x)(\z,0.6)(\y)};}

\foreach \x/\y/\z in 
{b1/b2/-6.8,b3/b4/-1.3,b5/b6/6.2}
{\draw[edge] plot [smooth,tension=1] coordinates{(\x)(\z,-0.6)(\y)};}

\foreach \x in
{-6.8,-1.3,6.2}
{\node[ver] () at (\x,-0.3){\tiny{$n-2$}}; }

\foreach \x in
{-6.8,-1.3,6.2}
{\node[ver] () at (\x,0.1){\tiny{$0,1,\dots,$}}; }

\node[ver] () at (-3.9,0.2){\tiny{$n-1$}};

\end{scope}

\begin{scope}[shift={(0,-5)}]
\foreach \x/\y/\z in
{-8/0/1,-5.5/0/2,-2.5/0/3,0/0/4,5/0/5,7.5/0/6}
{\node[vert] (c\z) at (\x,\y){};}

\foreach \x/\y in
{1.5/0,2.5/0,3.5/0}
{\node[extra] at (\x,\y){};}

\foreach \x/\y/\z in
{-8.7/-0.5/v_1^{n+1},-5/-0.5/v_2^{n+1},-3/-0.5/v_3^{n+1},0.5/-0.5/v_4^{n+1},4.3/-0.5/v_{2d-1}^{n+1},8.1/-0.5/v_{2d}^{n+1}}
{\node[ver] () at (\x,\y){$\z$};}

\draw[edge] (c2)--(c3);

\foreach \x/\y/\z in 
{c1/c2/-6.8,c3/c4/-1.3,c5/c6/6.2}
{\draw[edge] plot [smooth,tension=1] coordinates{(\x)(\z,0.6)(\y)};}

\foreach \x/\y/\z in 
{c1/c2/-6.8,c3/c4/-1.3,c5/c6/6.2}
{\draw[edge] plot [smooth,tension=1] coordinates{(\x)(\z,-0.6)(\y)};}

\foreach \x in
{-6.8,-1.3,6.2}
{\node[ver] () at (\x,0.3){\tiny{$n-1,$}}; }

\foreach \x in
{-6.8,-1.3,6.2}
{\node[ver] () at (\x,-0.1){\tiny{$0,\dots,n-3$}}; }

\node[ver] () at (-3.9,0.2){\tiny{$n-2$}};

\end{scope}
\foreach \x/\y/\z/\w in
{b1/a2/-5/3,b6/a3/5/3,c1/a8/-5/-3,c6/a9/5/-3}
{\draw[edge] plot[smooth,tension=1.5] coordinates{(\x) (\z,\w)(\y)};}

\foreach \x/\y in 
{-5/2.7,5/2.7}
{\node[ver] () at (\x,\y){\tiny{$n-1$}};}

\foreach \x/\y in 
{-5/-2.7,5/-2.7}
{\node[ver] () at (\x,\y){\tiny{$n-2$}};}

\end{tikzpicture}
\caption{A gem of $\mathbb{S}^{n-1}\times I$.}\label{fig:Ip}
\end{figure}

\begin{figure}[ht]
\tikzstyle{vert}=[circle, draw, fill=black!100, inner sep=0pt, minimum width=3pt]
\tikzstyle{ver}=[]
\tikzstyle{extra}=[circle, draw, fill=black!50, inner sep=0pt, minimum width=4pt]
\tikzstyle{edge} = [draw,thick,-]
\centering
\begin{tikzpicture}[scale=.8]

\begin{scope}[shift={(0,4)}]
\foreach \x/\y/\z in
{-8/0/1,-5.5/0/2,-2.5/0/3,0/0/4,5/0/5,7.5/0/6}
{\node[vert] (\z) at (\x,\y){};}

\foreach \x/\y in
{1.5/0,2.5/0,3.5/0}
{\node[extra] at (\x,\y){};}

\foreach \x/\y/\z in
{-8.5/-0.5/v_1^1,-5/-0.5/v_2^1,-3/-0.5/v_3^1,0.5/-0.5/v_4^1,4.3/-0.5/v_{2d-1}^1,8/-0.5/v_{2d}^1}
{\node[ver] () at (\x,\y){$\z$};}

\draw[edge] (2)--(3);

\foreach \x/\y/\z in 
{1/2/-6.8,3/4/-1.3,5/6/6.2}
{\draw[edge] plot [smooth,tension=1] coordinates{(\x)(\z,0.6)(\y)};}

\foreach \x/\y/\z in 
{1/2/-6.8,3/4/-1.3,5/6/6.2}
{\draw[edge] plot [smooth,tension=1] coordinates{(\x)(\z,-0.6)(\y)};}

\foreach \x in
{-6.8,-1.3,6.2}
{\node[ver] () at (\x,-0.3){\tiny{$n-2$}}; }

\foreach \x in
{-6.8,-1.3,6.2}
{\node[ver] () at (\x,0.1){\tiny{$0,1,\dots,$}}; }

\draw[edge] plot[smooth,tension=0.5] coordinates{(1) (-6.5,-1) (6,-1) (6)};

\node[ver] () at (0,-1.3){\tiny{$n-1$}};
\node[ver] () at (-3.9,0.2){\tiny{$n-1$}};

\node[ver] () at (0,-2){$(a)$};
\end{scope}

\begin{scope}[shift={(0,0)}]
\foreach \x/\y/\z in
{-8/0/1,-5.5/0/2,-2.5/0/3,0/0/4,5/0/5,7.5/0/6}
{\node[vert] (\z) at (\x,\y){};}

\foreach \x/\y in
{1.5/0,2.5/0,3.5/0}
{\node[extra] at (\x,\y){};}

\foreach \x/\y/\z in
{-8.7/-0.5/v_1^{n+1},-5/-0.5/v_2^{n+1},-3/-0.5/v_3^{n+1},0.5/-0.5/v_4^{n+1},4.3/-0.5/v_{2d-1}^{n+1},8.1/-0.5/v_{2d}^{n+1}}
{\node[ver] () at (\x,\y){$\z$};}

\draw[edge] (2)--(3);

\foreach \x/\y/\z in 
{1/2/-6.8,3/4/-1.3,5/6/6.2}
{\draw[edge] plot [smooth,tension=1] coordinates{(\x)(\z,0.6)(\y)};}

\foreach \x/\y/\z in 
{1/2/-6.8,3/4/-1.3,5/6/6.2}
{\draw[edge] plot [smooth,tension=1] coordinates{(\x)(\z,-0.6)(\y)};}

\foreach \x in
{-6.8,-1.3,6.2}
{\node[ver] () at (\x,0.3){\tiny{$n-1,$}}; }

\foreach \x in
{-6.8,-1.3,6.2}
{\node[ver] () at (\x,-0.1){\tiny{$0,\dots,n-3$}}; }

\draw[edge] plot[smooth,tension=0.5] coordinates{(1) (-6.5,-1) (6,-1) (6)};

\node[ver] () at (0,-1.3){\tiny{$n-2$}};
\node[ver] () at (-3.9,0.2){\tiny{$n-2$}};

\node[ver] () at (0,-2){$(b)$};
\end{scope}

\end{tikzpicture}

\caption{Two boundary components of $\mathbb{S}^{n-1}\times I$.}\label{fig:Ibd}
\end{figure}
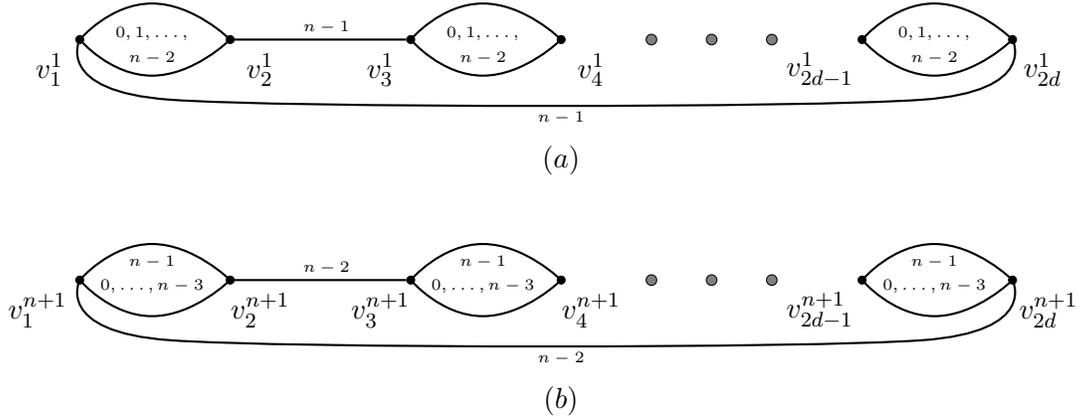

\noindent \textbf{\underline{Construction of a gem of $\mathbb{S}^{n-1}\times I$}:} Figure \ref{fig:crys} $(a)$ represents the standard crystallization of $\mathbb{S}^{n-1}\times I$, $n\ge 2$. Figure \ref{fig:I} exhibits an $(n+1)$-colored graph with boundary. Let $\Lambda_1^j=\{v_{2j}^2,v_{2j}^3,\dots, v_{2j}^n\}$ and $\Lambda_2^j=\{v_{2j+1}^2,v_{2j+1}^3,\dots, v_{2j+1}^n\}$ for all $1\le j \le d-1$. The map $\Phi_j:\Lambda_1^j\to \Lambda_2^j$, such that $\Phi(v_{2j}^k)=v_{2j+1}^k$ for all $2 \le k \le n$, induces an isomorphism between the subgraphs $A_1$ and $A_2$ generated by $\Lambda_1$ and $\Lambda_2$, respectively.
Then, applying polyhedral glue moves with respect to ($\Phi_j, \Lambda_1^j,  \Lambda_2^j, n$) for all $1\le j \le d-1$ successively, we get an $(n+1)$-colored graph with boundary as shown in Figure \ref{fig:Ip}. Now, canceling $1$-dipole with respect to color $n-1$ (resp. $n-2$) involving vertices $v_{2j}^1$ (resp. $v_{2j}^{n+1}$) and $v_{2j+1}^1$ (resp. $v_{2j+1}^{n+1}$) for $1 \le j \le d-1$ successively, we get the standard crystallization of $\mathbb{S}^{n-1}\times I$ (Figure \ref{fig:crys} (a)). Hence, Figure \ref{fig:I} exhibits a gem of $\mathbb{S}^{n-1}\times I$.

\begin{lemma} \label{Lemma}
    Let $\B_1$ and $\B_2$ be as shown in Figure $\ref{fig:Ibd}$ ($a$) and ($b$), respectively. Let $\B$ be an $(n+1)$-regular colored graph obtained from $\B_1$ and $\B_2$ in the following manner:
    \begin{itemize}
        \item  If $n$ is odd, then $v_1^1$ is connected to $v_1^{n+1}$, and for each $2 \le j \le 2d$, $v_j^1$ is connected to $v_{2d+2-j}^{n+1}$ by edges of color $n$.
        \item If $n$ is even, then $v_{2d}^1$ is connected to $v_1^{n+1}$, and for each $1 \le j \le 2d-1$, $v_j^1$ is connected to $v_{j+1}^{n+1}$ by edges of color $n$.
    \end{itemize}
    
 Then  $\B$  represents $\mathbb{S}^{n}$.
\end{lemma}

\begin{proof}
  The $n$-regular colored graphs $\B_1$ and $\B_2$ represent $\mathbb{S}^{n-1}$. The numbers of different bi-colored cycles in $\B$ are as follows: $g_{\{i,n\}}=1,\ g_{\{n-1,n\}}=g_{\{n-2,n\}}=d,\ g_{\{i,n-1\}}=d+1,\ g_{\{n-2,n-1\}}=2,\ g_{\{i,n-2\}}=d+1,\ g_{\{i,j\}}=2d$ for all $0\le i < j \le n-3$. 
  
  Therefore, the regular genus of $\B$ with respect to the permutation $(\varepsilon_0,\dots,\varepsilon_n)=(0,1,\dots,\\n-3,n-1,n,n-2)$, is given by $$\rho(\B)=1-\frac{1}{2}\Bigg( (1-n)\frac{4d}{2}+\sum_{k\in \mathbb{Z}_{n+1}} g_{\{\varepsilon_k,\varepsilon_{k+1}\}}\Bigg).$$ 
  Simplifying further, $$\rho(\B)=1-\frac{1}{2}\Big( (1-n)2d+2d(n-3)+(d+1)+d+d+(d+1)\Big)=0.$$ 

For every $i \in \{2, \dots, n\}$, each component of the subgraph of $\mathcal B$, induced by any $i+1$ colors, embeds regularly on the $2$-sphere, indicating that its regular genus is zero. According to \cite{fg82}, if an $(i+1)$-regular colored graph is a gem and has the regular genus zero, it represents the $i$-sphere. Since it is known that every 3-regular colored graph is a gem, each component of the subgraph of $\B$, induced by any $3$ colors, represents the $2$-sphere. By induction, each $(i+1)$-regular colored graph is a gem and thus, represents the $i$-sphere for all $i \in \{2, \dots, n\}$. In particular, the $(n+1)$-regular colored graph $\mathcal{B}$ is a gem and represents the $n$-sphere.
\end{proof}

 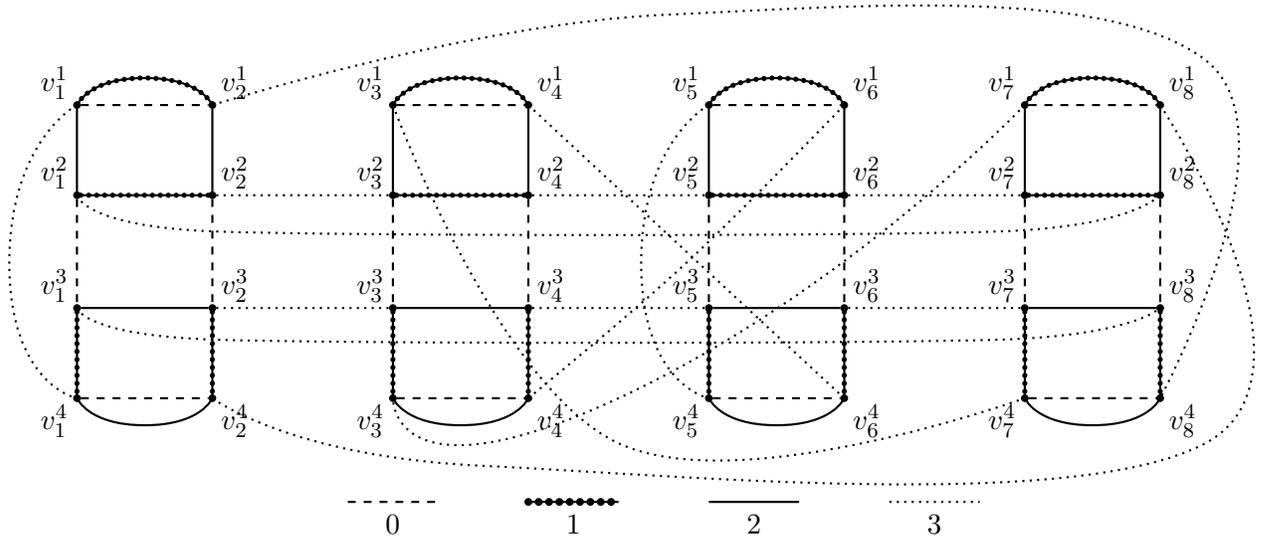
\begin{figure}[h!]
\tikzstyle{ver}=[]
\tikzstyle{verti}=[circle, draw, fill=black!100, inner sep=0pt, minimum width=2.5pt]
\tikzstyle{edge} = [draw,thick,-]
    \centering
\begin{tikzpicture}[scale=0.55]

\begin{scope}[shift={(-12,0)}]
\foreach \x/\y/\z in
{2/-2/w2,2/0/z2,2/2.5/y2,2/4.5/x2,-1/-2/w1,-1/0/z1,-1/2.5/y1,-1/4.5/x1}
{\node[verti] (\z) at (\x,\y){};}

\foreach \x/\y/\z in
{2.5/-2.5/v_2^4,2.5/0.5/v_2^3,2.5/3/v_2^2,2.5/5/v_2^1,-1.5/-2.5/v_1^4,-1.5/0.5/v_1^3,-1.5/3/v_1^2,-1.5/5/v_1^1}
{\node[ver] (\z) at (\x,\y){$\z$};}

\draw[edge] plot [smooth,tension=1.5] coordinates{(w1) (0.5,-2.6) (w2)};

\draw [line width=2pt, line cap=round, dash pattern=on 0pt off 1.5\pgflinewidth] plot [smooth,tension=1.5] coordinates{(x1) (0.5,5.1) (x2)};

\draw [edge] plot [smooth,tension=1.5] coordinates{(x1) (0.5,5.1) (x2)};

\end{scope}

\begin{scope}[shift={(-5,0)}]
\foreach \x/\y/\z in
{2/-2/w4,2/0/z4,2/2.5/y4,2/4.5/x4,-1/-2/w3,-1/0/z3,-1/2.5/y3,-1/4.5/x3}
{\node[verti] (\z) at (\x,\y){};}

\foreach \x/\y/\z in
{2.5/-2.5/v_4^4,2.5/0.5/v_4^3,2.5/3/v_4^2,2.5/5/v_4^1,-1.5/-2.5/v_3^4,-1.5/0.5/v_3^3,-1.5/3/v_3^2,-1.5/5/v_3^1}
{\node[ver] (\z) at (\x,\y){$\z$};}

\draw[edge] plot [smooth,tension=1.5] coordinates{(w3) (0.5,-2.6) (w4)};

\draw [line width=2pt, line cap=round, dash pattern=on 0pt off 1.5\pgflinewidth] plot [smooth,tension=1.5] coordinates{(x3) (0.5,5.1) (x4)};

\draw [edge] plot [smooth,tension=1.5] coordinates{(x3) (0.5,5.1) (x4)};
\end{scope}

\begin{scope}[shift={(2,0)}]
\foreach \x/\y/\z in
{2/-2/w6,2/0/z6,2/2.5/y6,2/4.5/x6,-1/-2/w5,-1/0/z5,-1/2.5/y5,-1/4.5/x5}
{\node[verti] (\z) at (\x,\y){};}

\foreach \x/\y/\z in
{2.5/-2.5/v_6^4,2.5/0.5/v_6^3,2.5/3/v_6^2,2.5/5/v_6^1,-1.5/-2.5/v_5^4,-1.5/0.5/v_5^3,-1.5/3/v_5^2,-1.5/5/v_5^1}
{\node[ver] (\z) at (\x,\y){$\z$};}

\draw[edge] plot [smooth,tension=1.5] coordinates{(w5) (0.5,-2.6) (w6)};

\draw [line width=2pt, line cap=round, dash pattern=on 0pt off 1.5\pgflinewidth] plot [smooth,tension=1.5] coordinates{(x5) (0.5,5.1) (x6)};

\draw [edge] plot [smooth,tension=1.5] coordinates{(x5) (0.5,5.1) (x6)};
\end{scope}

\begin{scope}[shift={(9,0)}]
\foreach \x/\y/\z in
{2/-2/w8,2/0/z8,2/2.5/y8,2/4.5/x8,-1/-2/w7,-1/0/z7,-1/2.5/y7,-1/4.5/x7}
{\node[verti] (\z) at (\x,\y){};}

\foreach \x/\y/\z in
{2.5/-2.5/v_8^4,2.5/0.5/v_8^3,2.5/3/v_8^2,2.5/5/v_8^1,-1.5/-2.5/v_7^4,-1.5/0.5/v_7^3,-1.5/3/v_7^2,-1.5/5/v_7^1}
{\node[ver] (\z) at (\x,\y){$\z$};}

\draw[edge] plot [smooth,tension=1.5] coordinates{(w7) (0.5,-2.6) (w8)};

\draw [line width=2pt, line cap=round, dash pattern=on 0pt off 1.5\pgflinewidth] plot [smooth,tension=1.5] coordinates{(x7) (0.5,5.1) (x8)};

\draw [edge] plot [smooth,tension=1.5] coordinates{(x7) (0.5,5.1) (x8)};
\end{scope}

\foreach \x/\y in
{w1/w2,w3/w4,w5/w6,w7/w8,y1/z1,y2/z2,y3/z3,y4/z4,y5/z5,y6/z6,y7/z7,y8/z8,x1/x2,x3/x4,x5/x6,x7/x8}
{\path[edge,dashed] (\x) -- (\y);}

\foreach \x/\y in
{y1/x1,y2/x2,y3/x3,y4/x4,y5/x5,y6/x6,y7/x7,y8/x8,z1/w1,z2/w2,z3/w3,z4/w4,z5/w5,z6/w6,z7/w7,z8/w8,y1/y2,y3/y4,y5/y6,y7/y8,z1/z2,z3/z4,z5/z6,z7/z8}
{\path[edge] (\x) -- (\y);}

\foreach \x/\y in {z1/w1,z2/w2,z3/w3,z4/w4,z5/w5,z6/w6,z7/w7,z8/w8,y1/y2,y3/y4,y5/y6,y7/y8}
{{\draw [line width=2pt, line cap=round, dash pattern=on 0pt off 1.5\pgflinewidth]  (\x) -- (\y);}}

\foreach \x/\y in 
{y2/y3,y4/y5,y6/y7,z2/z3,z4/z5,z6/z7}
{\path[edge,dotted] (\x) -- (\y);}

\draw[edge, dotted] plot [smooth,tension=0.5] coordinates{(y1) (-10,1.7) (7.5,1.7) (y8)};

\draw[edge, dotted] plot [smooth,tension=0.5] coordinates{(z1) (-10,-0.7) (7.5,-0.7) (z8)};

\draw[edge, dotted] plot [smooth,tension=1.5] coordinates{(x1) (-14.5,1) (w1)};

\draw[edge, dotted] plot [smooth,tension=1.5] coordinates{(x5) (-0.5,1) (w5)};

\draw[edge, dotted] plot [smooth,tension=0.7] coordinates{(x2) (0,6.5) (12,5.5) (w8)};

\draw[edge, dotted] plot [smooth,tension=0.7] coordinates{(x3) (-1,-3)  (w7)};

\draw[edge, dotted] plot [smooth,tension=1] coordinates{(x4)  (w6)};

\draw[edge, dotted] plot [smooth,tension=0.7] coordinates{(x6) (0.5, 1)  (w4)};

\draw[edge, dotted] plot [smooth,tension=1.3] coordinates{(x7) (-1,-2)  (w3)};

\draw[edge, dotted] plot [smooth,tension=0.7] coordinates{(x8)  (12,-3) (-4,-3.5) (w2)};

\begin{scope} [scale=1, shift = {(-8, -4.7)}]
\foreach \x/\y/\z in {1/-0.5/0,5/-0.5/1,9/-0.5/2,13/-0.5/3}
{\node[ver] () at (\x,\y){$\z$};}
\path[edge,dashed] (0,0) -- (2,0);
\path[edge] (4,0) -- (6,0);

\path[edge] (8,0) -- (10,0);
\draw [line width=2pt, line cap=round, dash pattern=on 0pt off 1.3\pgflinewidth]  (4,0) -- (6,0);

\path[edge,dotted] (12,0) -- (14,0);
\end{scope}

\end{tikzpicture}

\caption{A gem of $\mathbb{S}^{2}\times \mathbb{S}^{1}$ corresponding to degree $4$.}\label{fig:ex1}
\end{figure}

 \begin{figure}[h!]
\tikzstyle{ver}=[]
\tikzstyle{verti}=[circle, draw, fill=black!100, inner sep=0pt, minimum width=3pt]
\tikzstyle{edge} = [draw,thick,-]
    \centering
\begin{tikzpicture}[scale=0.55]

\begin{scope}[shift={(-12,0)}]
\foreach \x/\y/\z in
{2/-4/u2,2/-2/w2,2/0/z2,2/2.5/y2,2/4.5/x2,-1/-4/u1,-1/-2/w1,-1/0/z1,-1/2.5/y1,-1/4.5/x1}
{\node[verti] (\z) at (\x,\y){};}

\foreach \x/\y/\z in
{2.5/-4.5/v_2^5,2.5/-2.5/v_2^4,2.5/-0.5/v_2^3,2.5/3/v_2^2,2.5/5/v_2^1,-1.5/-4.5/v_1^5,-1.5/-2.5/v_1^4,-1.5/-0.5/v_1^3,-1.5/3/v_1^2,-1.5/5/v_1^1}
{\node[ver] (\z) at (\x,\y){$\z$};}

\draw[edge,dashed] plot [smooth,tension=1.5] coordinates{(u1) (0.5,-4.6) (u2)};

\draw [line width=2pt, line cap=round, dash pattern=on 0pt off 1.5\pgflinewidth] plot [smooth,tension=1.5] coordinates{(u1) (0.5,-3.4) (u2)};

\draw [edge] plot [smooth,tension=1.5] coordinates{(u1) (0.5,-3.4) (u2)};

\draw[edge, dashed] plot [smooth,tension=1.5] coordinates{(w1) (0.5,-2.6) (w2)};

\draw[edge] plot [smooth,tension=1.5] coordinates{(z1) (0.5,-0.6) (z2)};

\draw[edge] plot [smooth,tension=1.5] coordinates{(y1) (0.5,3.1) (y2)};

\draw[edge] plot [smooth,tension=1.5] coordinates{(x1) (0.5,4) (x2)};

\draw [line width=2pt, line cap=round, dash pattern=on 0pt off 1.5\pgflinewidth] plot [smooth,tension=1.5] coordinates{(x1) (0.5,5.1) (x2)};

\draw [edge] plot [smooth,tension=1.5] coordinates{(x1) (0.5,5.1) (x2)};

\end{scope}

\begin{scope}[shift={(-5,0)}]
\foreach \x/\y/\z in
{2/-4/u4,2/-2/w4,2/0/z4,2/2.5/y4,2/4.5/x4,-1/-4/u3,-1/-2/w3,-1/0/z3,-1/2.5/y3,-1/4.5/x3}
{\node[verti] (\z) at (\x,\y){};}

\foreach \x/\y/\z in
{2.5/-4.5/v_4^5,2.5/-2.5/v_4^4,2.5/-0.5/v_4^3,2.5/3/v_4^2,2.5/5/v_4^1,-1.5/-4.5/v_3^5,-1.5/-2.5/v_3^4,-1.5/-0.5/v_3^3,-1.5/3/v_3^2,-1.5/5/v_3^1}
{\node[ver] (\z) at (\x,\y){$\z$};}

\draw[edge,dashed] plot [smooth,tension=1.5] coordinates{(u3) (0.5,-4.6) (u4)};

\draw [line width=2pt, line cap=round, dash pattern=on 0pt off 1.5\pgflinewidth] plot [smooth,tension=1.5] coordinates{(u3) (0.5,-3.4) (u4)};

\draw [edge] plot [smooth,tension=1.5] coordinates{(u3) (0.5,-3.4) (u4)};

\draw[edge, dashed] plot [smooth,tension=1.5] coordinates{(w3) (0.5,-2.6) (w4)};

\draw[edge] plot [smooth,tension=1.5] coordinates{(z3) (0.5,-0.6) (z4)};

\draw[edge] plot [smooth,tension=1.5] coordinates{(y3) (0.5,3.1) (y4)};

\draw[edge] plot [smooth,tension=1.5] coordinates{(x3) (0.5,4) (x4)};

\draw [line width=2pt, line cap=round, dash pattern=on 0pt off 1.5\pgflinewidth] plot [smooth,tension=1.5] coordinates{(x3) (0.5,5.1) (x4)};

\draw [edge] plot [smooth,tension=1.5] coordinates{(x3) (0.5,5.1) (x4)};

\end{scope}

\begin{scope}[shift={(2,0)}]
\foreach \x/\y/\z in
{2/-4/u6,2/-2/w6,2/0/z6,2/2.5/y6,2/4.5/x6,-1/-4/u5,-1/-2/w5,-1/0/z5,-1/2.5/y5,-1/4.5/x5}
{\node[verti] (\z) at (\x,\y){};}

\foreach \x/\y/\z in
{2.5/-4.5/v_6^5,2.5/-2.5/v_6^4,2.5/-0.5/v_6^3,2.5/3/v_6^2,2.5/5/v_6^1,-1.5/-4.5/v_5^5,-1.5/-2.5/v_5^4,-1.5/-0.5/v_5^3,-1.5/3/v_5^2,-1.5/5/v_5^1}
{\node[ver] (\z) at (\x,\y){$\z$};}

\draw[edge,dashed] plot [smooth,tension=1.5] coordinates{(u5) (0.5,-4.6) (u6)};

\draw [line width=2pt, line cap=round, dash pattern=on 0pt off 1.5\pgflinewidth] plot [smooth,tension=1.5] coordinates{(u5) (0.5,-3.4) (u6)};

\draw [edge] plot [smooth,tension=1.5] coordinates{(u5) (0.5,-3.4) (u6)};

\draw[edge, dashed] plot [smooth,tension=1.5] coordinates{(w5) (0.5,-2.6) (w6)};

\draw[edge] plot [smooth,tension=1.5] coordinates{(z5) (0.5,-0.6) (z6)};

\draw[edge] plot [smooth,tension=1.5] coordinates{(y5) (0.5,3.1) (y6)};

\draw[edge] plot [smooth,tension=1.5] coordinates{(x5) (0.5,4) (x6)};

\draw [line width=2pt, line cap=round, dash pattern=on 0pt off 1.5\pgflinewidth] plot [smooth,tension=1.5] coordinates{(x5) (0.5,5.1) (x6)};

\draw [edge] plot [smooth,tension=1.5] coordinates{(x5) (0.5,5.1) (x6)};

\end{scope}

\begin{scope}[shift={(9,0)}]
\foreach \x/\y/\z in
{2/-4/u8,2/-2/w8,2/0/z8,2/2.5/y8,2/4.5/x8,-1/-4/u7,-1/-2/w7,-1/0/z7,-1/2.5/y7,-1/4.5/x7}
{\node[verti] (\z) at (\x,\y){};}

\foreach \x/\y/\z in
{2.5/-4.5/v_8^5,2.5/-2.5/v_8^4,2.5/-0.5/v_8^3,2.5/3/v_8^2,2.5/5/v_8^1,-1.5/-4.5/v_7^5,-1.5/-2.5/v_7^4,-1.5/-0.5/v_7^3,-1.5/3/v_7^2,-1.5/5/v_7^1}
{\node[ver] (\z) at (\x,\y){$\z$};}

\draw[edge, dashed] plot [smooth,tension=1.5] coordinates{(u7) (0.5,-4.6) (u8)};

\draw [line width=2pt, line cap=round, dash pattern=on 0pt off 1.5\pgflinewidth] plot [smooth,tension=1.5] coordinates{(u7) (0.5,-3.4) (u8)};

\draw [edge] plot [smooth,tension=1.5] coordinates{(u7) (0.5,-3.4) (u8)};

\draw[edge, dashed] plot [smooth,tension=1.5] coordinates{(w7) (0.5,-2.6) (w8)};

\draw[edge] plot [smooth,tension=1.5] coordinates{(z7) (0.5,-0.6) (z8)};

\draw[edge] plot [smooth,tension=1.5] coordinates{(y7) (0.5,3.1) (y8)};

\draw[edge] plot [smooth,tension=1.5] coordinates{(x7) (0.5,4) (x8)};

\draw [line width=2pt, line cap=round, dash pattern=on 0pt off 1.5\pgflinewidth] plot [smooth,tension=1.5] coordinates{(x7) (0.5,5.1) (x8)};

\draw [edge] plot [smooth,tension=1.5] coordinates{(x7) (0.5,5.1) (x8)};

\end{scope}

\foreach \x/\y in
{y1/z1,y2/z2,y3/z3,y4/z4,y5/z5,y6/z6,y7/z7,y8/z8,x1/x2,x3/x4,x5/x6,x7/x8}
{\path[edge,dashed] (\x) -- (\y);}

\foreach \x/\y in
{w1/u1,w2/u2,w3/u3,w4/u4,w5/u5,w6/u6,w7/u7,w8/u8,z1/w1,z2/w2,z3/w3,z4/w4,z5/w5,z6/w6,z7/w7,z8/w8,y1/y2,y3/y4,y5/y6,y7/y8}
{\path[edge] (\x) -- (\y);}

\foreach \x/\y in {z1/w1,z2/w2,z3/w3,z4/w4,z5/w5,z6/w6,z7/w7,z8/w8,y1/y2,y3/y4,y5/y6,y7/y8}
{{\draw [line width=2pt, line cap=round, dash pattern=on 0pt off 1.5\pgflinewidth]  (\x) -- (\y);}}

\foreach \x/\y in 
{y2/y3,y4/y5,y6/y7,z2/z3,z4/z5,z6/z7,w2/w3,w4/w5,w6/w7,x1/u2,x2/u3,x3/u4,x4/u5,x5/u6,x6/u7,x7/u8}
{\path[edge,dotted] (\x) -- (\y);}

\foreach \x/\y in 
{x1/y1,x2/y2,x3/y3,x4/y4,x5/y5,x6/y6,x7/y7,x8/y8,z1/z2,z3/z4,z5/z6,z7/z8,w1/w2,w3/w4,w5/w6,w7/w8,u1/u2,u3/u4,u5/u6,u7/u8}
{\draw[decorate,decoration={snake, amplitude=1pt, segment length=8pt}] (\x) -- (\y);}

\draw[edge, dotted] plot [smooth,tension=0.5] coordinates{(y1) (-10,1.7) (7.5,1.7) (y8)};

\draw[edge, dotted] plot [smooth,tension=0.5] coordinates{(z1) (-10,0.8) (7.5,0.8) (z8)};

\draw[edge, dotted] plot [smooth,tension=0.5] coordinates{(w1) (-10,-1.2) (7.5,-1.2) (w8)};

\draw[edge, dotted] plot [smooth,tension=0.9] coordinates{(x8)  (12,-5) (-7,-5.5) (u1)};

\begin{scope} [scale=1, shift = {(-9, -7)}]
\foreach \x/\y/\z in {1/-0.5/0,5/-0.5/1,9/-0.5/2,13/-0.5/3,17/-0.5/4}
{\node[ver] () at (\x,\y){$\z$};}
\path[edge,dashed] (0,0) -- (2,0);
\path[edge] (4,0) -- (6,0);
\draw [line width=2pt, line cap=round, dash pattern=on 0pt off 1.3\pgflinewidth]  (4,0) -- (6,0);
\path[edge] (8,0) -- (10,0);
\draw[decorate,decoration={snake, amplitude=1pt, segment length=8pt}] (12,0) -- (14,0);

\path[edge,dotted] (16,0) -- (18,0);
\end{scope}

\end{tikzpicture}

\caption{A gem of $\mathbb{S}^{3}\times \mathbb{S}^{1}$ corresponding to degree $4$.}\label{fig:ex2}
\end{figure}
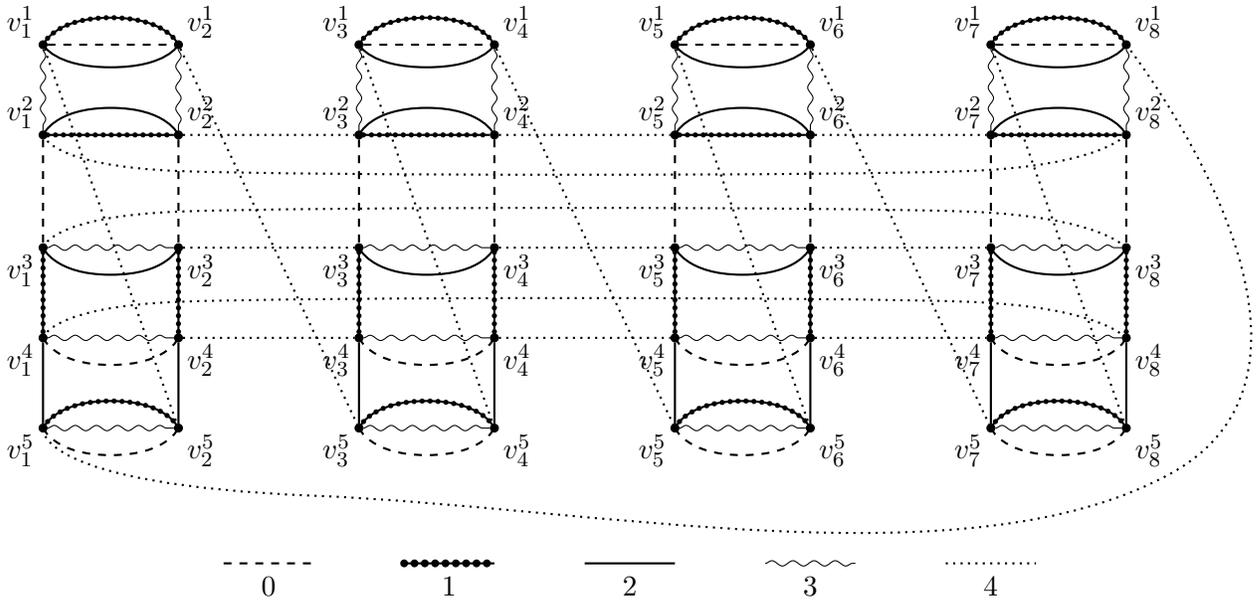

\begin{theorem}\label{theorem:main2}
   For each $d \in \mathbb{Z}$, there exists a minimal simplicial degree $d$ self-map of $\mathbb{S}^{n-1}\times \mathbb{S}^1$ with respect to the standard triangulation of $\mathbb{S}^{n-1}\times \mathbb{S}^1$, where $n\geq 2$.  
\end{theorem}

\begin{proof}
From Subsection \ref{crystal}, it follows that Figure \ref{fig:Ibd} exhibits two boundary components of $\mathbb{S}^{n-1}\times I$ (Figure \ref{fig:I}). The boundary components $\B_1$ and $\B_2$ (Figure~\ref{fig:Ibd} (a) and (b), respectively) both represent $\mathbb{S}^{n-1}$ as $n$-regular colored graphs. Equivalently, both $\B_1$ and $\B_2$ represent the $n$-ball $\mathbb{B}^{n}$ as $(n+1)$-colored graphs with connected boundaries. Let $\B$ be an $(n+1)$-regular colored graph obtained from $\B_1$ and $\B_2$ in the following manner:
    \begin{itemize}
        \item  If $n$ is odd, then $v_1^1$ is connected to $v_1^{n+1}$, and for each $2 \le j \le 2d$, $v_j^1$ is connected to $v_{2d+2-j}^{n+1}$  by edges of color $n$.
        \item If $n$ is even, then $v_{2d}^1$ is connected to $v_1^{n+1}$, and for each $1 \le j \le 2d-1$, $v_j^1$ is connected to $v_{j+1}^{n+1}$ by edges of color $n$.
    \end{itemize}

\noindent Then, by Lemma \ref{Lemma},  $(n+1)$-regular colored graph $\B$ represents the $n$-sphere $\mathbb{S}^{n}$. Consequently, the edges of color $n$ establish identifications between the boundaries of two $n$-balls, represented by $\B_1$ and $\B_2$, respectively. By maintaining these edges of color $n$ between the boundary components of the gem of $\mathbb{S}^{n-1}\times I$ (Figure \ref{fig:I}), the boundaries of $\mathbb{S}^{n-1}\times I$ are identified, resulting in an $\mathbb{S}^{n-1}$-bundle over $\mathbb{S}^{1}$. Since the gem is bipartitite, it represents $\mathbb{S}^{n-1}\times \mathbb{S}^{1}$ (cf. \cite{gv87, S51}). Figures \ref{fig:ex1} and \ref{fig:ex2}  illustrate gems of $\mathbb{S}^{2}\times \mathbb{S}^{1}$ and $\mathbb{S}^{3}\times \mathbb{S}^{1}$, respectively,  corresponding to degree $4$. 

Now, let $\G_1$ be the gem of $\mathbb{S}^{n-1}\times \mathbb{S}^{1}$ obtained as described above, and let $\G_2$ be the standard crystallization of $\mathbb{S}^{n-1}\times \mathbb{S}^{1}$ (Figure \ref{fig:crys} ($b$) if $n$ is odd and Figure \ref{fig:crys} ($c$) if $n$ is even). Let $v_{2j-1}^{k}$ (resp. $v_{2j}^{k}$), $1 \le j \le d$ and $1 \le k \le n+1$, be positive (resp. negative) vertices in Figure \ref{fig:I} and \ref{fig:crys} ($b,c$). Define $g_0:\G_2 \to \G_2$ such that $g_0$ is a constant function. Clearly, $g_0$ induces a simplicial self-map of $\mathbb{S}^{n-1}\times \mathbb{S}^1$ of degree $0$, and this induced map is minimal since $\G_2$ is the standard crystallization.

Now, let $d\ge 1$.  Define $g_d:\G_1\to \G_2$ such that $g_d(v_{2j-1}^{k})=v_1^k$ and $g_d(v_{2j}^{k})=v_2^k$ for all $1 \le j \le d$ and $1 \le k \le n+1$. This map $g_d$ satisfies the property that $g_d(u)$ and $g_d(v)$ are either the same for any two vertices $u$ and $v$ in $\G_1$ or are joined by an edge of color $i$ in $\G_2$ whenever $u$ and $v$ are joined by an edge of color $i$ in $\G_1$. Therefore, the map $g_d$ induces a simplicial self-map $f$ of $\mathbb{S}^{n-1}\times \mathbb{S}^{1}$ of degree $d$ (cf. Lemma \ref{gmap}). These induced simplicial self-maps are minimal, since in the inverse image of any positive (resp. negative) vertex, there are exactly $d$ positive (resp. negative) vertices, and $\G_2$ is the standard crystallization. Defining $g':\G_2\to \G_2$ as $g'(v_1^k)=v_2^k$ for all $1\le k \le n+1$ and letting $g_{-d}=g' \circ g_d$, we get that $g_{-d}$ induces a simplicial self-map of $\mathbb{S}^{n-1}\times \mathbb{S}^1$ of degree $-d$, and since $g_d$ is minimal, it is minimal.
\end{proof}

Since a gem with $2p$ vertices of a closed connected $n$-manifold $M$ corresponds to a $2p$-facet colored triangulation of the $n$-manifold $M$, the constructions in Corollary \ref{corollary:main} and Theorems \ref{theorem:main1}, \ref{theorem:main2} yield the following result.

\begin{corollary}\label{corollary:main2}
\begin{enumerate}
\item[$(a)$] For each $ d \in \mathbb{Z} $, there exists a degree $ d $ simplicial map from a \\ $(2\max\{|d|,1\})$-facet colored triangulation of $ \mathbb{S}^n $ to the $ 2 $-facet colored triangulation of $ \mathbb{S}^n $. This configuration represents the minimal possible number of facets for a degree $ d $ simplicial self-map of $ \mathbb{S}^n $, where $ n \geq 1 $.

\item[$(b)$] For each $ d \in \mathbb{Z} $, there exists a degree $ d $ simplicial map from a $ (2(n+1)\max\{|d|,1\} )$-facet  colored triangulation of $ \mathbb{S}^{n-1} \times \mathbb{S}^1 $ to the standard $ 2(n+1) $-facet colored triangulation of $ \mathbb{S}^{n-1} \times \mathbb{S}^1 $. This configuration represents the minimal possible number of facets for a degree $ d $ simplicial self-map of $ \mathbb{S}^{n-1} \times \mathbb{S}^1 $ with respect to the standard triangulation of $ \mathbb{S}^{n-1} \times \mathbb{S}^1 $, where $ n \geq 2 $.
  \end{enumerate}
\end{corollary}

\noindent {\bf Acknowledgment:}  The authors thank the anonymous referees for many useful comments and suggestions that have greatly improved the article. The first author is supported by the Institute fellowship from the Indian Institute of Technology Delhi. The second author is supported by the Mathematical Research Impact Centric Support (MATRICS) Research Grant (MTR/2022/000036) by ANRF (India). The third author is supported by the National Research Foundation of Korea (NRF) grant funded by the Korean government (MSIT) (RS-2025-00555914). 
 
 \medskip

\noindent {\bf Data Availability:} No data was used for the research described in the article.

\medskip
 
\noindent {\bf Conflict of interest:} There is no competing interest.

\end{document}